\documentclass[a4paper,12pt]{amsart}

\usepackage[top=3cm,left=2.5cm,right=2.5cm,bottom=3cm]{geometry}
\usepackage{relsize}
\usepackage{lineno}
\usepackage{hyperref}
\usepackage{amsmath,amsthm,pb-diagram,amssymb,comment}
\usepackage{amsfonts,graphicx,color}
\usepackage{enumitem, fancyhdr, dsfont}
\usepackage{thmtools,cleveref}
\usepackage[normalem]{ulem}
\usepackage{stmaryrd}
\usepackage{textcomp}
\usepackage{contour}
\usepackage{fontawesome}
\usepackage{tikz,float}
\usepackage{chemformula}
\usepackage{amsthm,thmtools,xcolor} 

 \usepackage{multicol}
\usepackage{marvosym}
\hypersetup{
    colorlinks=true,
    linkcolor=dodger,
    filecolor=dodger,
    urlcolor=dodger,
    citecolor=dodger,
}

    \newcommand{\dom}{\mbox{\rm dom}}

    \newcommand{\ran}{\mbox{\rm ran}}

    \newcommand{\thzfc}{\mathrm{ZFC}}

    \newcommand{\Mbf}{\mathbf{M}}

    \newcommand{\Cn}{\mathbf{Cn}}
    
    \newcommand{\Awf}{\mathcal{A}}

    \newcommand{\Ewf}{\mathcal{E}}

    \newcommand{\Iwf}{\mathcal{I}}
    \newcommand{\Jwf}{\mathcal{J}}
    \newcommand{\Mwf}{\mathcal{M}}
    \newcommand{\Nwf}{\mathcal{N}}
    
    \newcommand{\Pwf}{\mathcal{P}}
    \newcommand{\Swf}{\mathcal{S}}

    \newcommand{\bfrak}{\mathfrak{b}}
    \newcommand{\cfrak}{\mathfrak{c}}
    \newcommand{\dfrak}{\mathfrak{d}}
    \newcommand{\efrak}{\mathfrak{e}}
    \newcommand{\eubd}{\mathfrak{e}_{\mathbf{ubd}}}
     \newcommand{\efin}{\mathfrak{e}_{\mathbf{fin}}}

    \newcommand{\sfrak}{\mathfrak{s}}

    \newcommand{\menos}{\smallsetminus}
    
    \newcommand{\pts}{\mathcal{P}}
    
    \newcommand{\frestr}{\!\!\upharpoonright\!\!}

    \newcommand{\add}{\mbox{\rm add}}
    \newcommand{\cov}{\mbox{\rm cov}}
    \newcommand{\non}{\mbox{\rm non}}
    \newcommand{\cof}{\mbox{\rm cof}}
    \newcommand{\limdir}{\mbox{\rm limdir}}

    \newcommand{\Bor}{\mathds{B}}

    \newcommand{\Eor}{\mathds{E}}

    \newcommand{\LOCor}{\mathds{LOC}}
    \newcommand{\Mor}{\mathds{M}}
    
    \newcommand{\Por}{\mathds{P}}
    \newcommand{\Pbb}{\mathds{P}}
    \newcommand{\Qor}{\mathds{Q}}
    \newcommand{\Ior}{\mathds{I}}

    \newcommand{\Ror}{\mathds{R}}

    \newcommand{\Pror}{\mathbb{PR}}

    \newcommand{\Qnm}{\dot{\mathds{Q}}}

    \newcommand{\R}{\mathbb{R}}

    \newcommand{\cf}{\mbox{\rm cf}}

    \newcommand{\sii}{{\ \mbox{$\Leftrightarrow$} \ }}

    \newcommand{\la}{\langle}
    \newcommand{\ra}{\rangle}

\newcommand{\minLc}{\mathrm{minLc}}

\newcommand{\Seq}{\mathrm{seq}}
\newcommand{\Fr}{\mathrm{Fr}}
\newcommand{\Rbf}{\mathbf{R}}

\newcommand{\Cbf}{\mathbf{C}}
\newcommand{\Lc}{\mathbf{Lc}}

\newcommand{\Lb}{\mathbf{Lb}}

\newcommand{\Scal}{\mathcal{S}}
\newcommand{\id}{\mathrm{id}}
\newcommand{\blc}{\mathfrak{b}^{\mathrm{Lc}}}
\newcommand{\dlc}{\mathfrak{d}^{\mathrm{Lc}}}
\newcommand{\balc}{\mathfrak{b}^{\mathrm{aLc}}}

\newcommand{\blcp}{\mathfrak{b}^{\mathrm{Lc_0}}}

\newcommand{\leqT}{\preceq_{\mathrm{T}}}
\newcommand{\eqT}{\cong_{\mathrm{T}}}

\newcommand{\gen}{\mathrm{gen}}

\newcommand{\DMor}{\mathbb{D}\mathbb{M}}

\newcommand{\sqsubm}{\sqsubset^{\rm m}}
\newcommand{\st}{\mid}
\newcommand{\set}[2]{\{#1 \st\, #2\}}
\newcommand{\largeset}[2]{\left\{#1 \;\middle|\; #2\right\}}
\newcommand{\seq}[2]{\la #1 \st #2\ra}

\newcommand{\baire}{\omega^\omega}

\newcommand{\cantor}{2^\omega}

\newcommand\subsetdot{\mathrel{\ooalign{$\subset$\cr
  \hidewidth\hbox{$\cdot\mkern3mu$}\cr}}}

\contourlength{0.8pt}

\contourlength{0.8pt}

	\definecolor{ultramarineblue}{rgb}{0.25, 0.4, 0.96}
\definecolor{cornellred}{rgb}{0.7, 0.11, 0.11}
\definecolor{cobalt}{rgb}{0.0, 0.28, 0.67}
\definecolor{bleudefrance}{rgb}{0.19, 0.55, 0.91}
\definecolor{darkblue}{rgb}{0.0, 0.0, 0.55}
\definecolor{ferrarired}{rgb}{1.0, 0.11, 0.0}
\definecolor{brandeisblue}{rgb}{0.0, 0.44, 1.0}
\definecolor{azure(colorwheel)}{rgb}{0.0, 0.5, 1.0}
\definecolor{aqua}{rgb}{0.0, 1.0, 1.0}
\definecolor{aguamarina}{cmyk}{0.85,0,0.33,0}
\definecolor{cafe}{cmyk}{0,0.81,1,0.60}
\definecolor{canela}{cmyk}{0.14,0.42,0.56,0}
\definecolor{darkgray}{cmyk}{0,0,0,0.50}
\definecolor{emerald}{cmyk}{0.91,0,0.88,0.12}
\definecolor{fresa}{cmyk}{0,1,0.50,0}
\definecolor{gold}{cmyk}{0,0.10,0.84,0}
\definecolor{lightgray}{cmyk}{0,0,0,0.30}
\definecolor{marron}{cmyk}{0,0.72,1,0.45}
\definecolor{melon}{cmyk}{0,0.29,0.84,0}
\definecolor{ladri}{cmyk}{0,0.77,0.87,0}
\definecolor{olive}{cmyk}{0.64,0,0.95,0.40}
\definecolor{orange}{cmyk}{0,0.42,1,0}
\definecolor{peach}{cmyk}{0,0.46,0.50,0}
\definecolor{pink}{cmyk}{0,0.10,0.10,0}
\definecolor{orange}{cmyk}{0,0.42,1,0}
\definecolor{pine}{cmyk}{0.92,0,0.59,0.25}
\definecolor{purple}{cmyk}{0.45,0.86,0,0}
\definecolor{violet}{cmyk}{0.07,0.90,0,0.34}
\definecolor{craneorange}{RGB}{252,187,6}
\definecolor{red(ncs)}{rgb}{0.77, 0.01, 0.2}

\definecolor{aguamarina}{cmyk}{0.85,0,0.33,0}
\definecolor{cafe}{cmyk}{0,0.81,1,0.60}
\definecolor{canela}{cmyk}{0.14,0.42,0.56,0}
\definecolor{darkgray}{cmyk}{0,0,0,0.50}
\definecolor{emerald}{cmyk}{0.91,0,0.88,0.12}
\definecolor{fresa}{cmyk}{0,1,0.50,0}
\definecolor{gold}{cmyk}{0,0.10,0.84,0}
\definecolor{lightgray}{cmyk}{0,0,0,0.30}
\definecolor{marron}{cmyk}{0,0.72,1,0.45}
\definecolor{melon}{cmyk}{0,0.29,0.84,0}
\definecolor{ladri}{cmyk}{0,0.77,0.87,0}
\definecolor{olive}{cmyk}{0.64,0,0.95,0.40}
\definecolor{orange}{cmyk}{0,0.42,1,0}
\definecolor{peach}{cmyk}{0,0.46,0.50,0}
\definecolor{pink}{cmyk}{0,0.10,0.10,0}
\definecolor{orange}{cmyk}{0,0.42,1,0}
\definecolor{pine}{cmyk}{0.92,0,0.59,0.25}
\definecolor{purple}{cmyk}{0.45,0.86,0,0}
\definecolor{violet}{cmyk}{0.07,0.90,0,0.34}
\definecolor{ceruleanblue}{rgb}{0.16, 0.32, 0.75}

\DeclareSymbolFont{extraup}{U}{zavm}{m}{n}
\DeclareMathSymbol{\varheart}{\mathalpha}{extraup}{86}
\DeclareMathSymbol{\vardiamond}{\mathalpha}{extraup}{87}

\definecolor{dodger}{rgb}{0.0,0.5,1.0}

\definecolor{amber}{rgb}{1.0,0.49,0.0}

\definecolor{ogreen}{RGB}{107,142,35}

\title[Soft-linkedness]{Soft-linkedness}

\author{Miguel A.~Cardona}
\address{Einstein Institute of Mathematics,
The Hebrew University of Jerusalem. Givat Ram, Jerusalem, 9190401, Israel}
\email{\href{mailto:miguel.cardona@mail.huji.ac.il}{miguel.cardona@mail.huji.ac.il}}
\urladdr{\url{https://sites.google.com/mail.huji.ac.il/miguel-cardona-montoya/home-page}}

\thanks{The author was supported by the Slovak Research and Development Agency under Contract No. APVV-20-0045; by Pavol Jozef \v{S}af\'arik University at a postdoctoral position; and by Israel Science Foundation for partial support of this research by grant 2320/23 (2023-2027).}



\date{}

\definecolor{sub0}{RGB}{29,32,137}
\definecolor{sub1}{RGB}{1,71,157}
\definecolor{sub2}{RGB}{1,104,183}
\definecolor{sub3}{RGB}{0,160,234}
\definecolor{sug}{RGB}{0,154,68}
\definecolor{suy}{RGB}{208,219,1}

\begin{document}

\makeatletter
\def\@roman#1{\romannumeral #1}
\makeatother

\newcounter{enuAlph}
\renewcommand{\theenuAlph}{\Alph{enuAlph}}

\numberwithin{equation}{section}
\renewcommand{\theequation}{\thesection.\arabic{equation}}

\theoremstyle{plain}
  \newtheorem{theorem}[equation]{Theorem}
  \newtheorem{corollary}[equation]{Corollary}
  \newtheorem{lemma}[equation]{Lemma}
  \newtheorem{mainlemma}[equation]{Main Lemma}
  \newtheorem*{mainthm}{Main Theorem}
  \newtheorem{prop}[equation]{Proposition}
  \newtheorem{clm}[equation]{Claim}
  \newtheorem{fact}[equation]{Fact}
  \newtheorem{exer}[equation]{Exercise}
  \newtheorem{question}[equation]{Question}
  \newtheorem{problem}[equation]{Problem}
  \newtheorem{conjecture}[equation]{Conjecture}
  \newtheorem{assumption}[equation]{Assumption}
  \newtheorem*{thm}{Theorem}
  \newtheorem{teorema}[enuAlph]{Theorem}
  \newtheorem*{corolario}{Corollary}
\theoremstyle{definition}
  \newtheorem{definition}[equation]{Definition}
  \newtheorem{example}[equation]{Example}
  \newtheorem{remark}[equation]{Remark}
  \newtheorem{notation}[equation]{Notation}
  \newtheorem{context}[equation]{Context}

  \newtheorem*{defi}{Definition}
  \newtheorem*{acknowledgements}{Acknowledgements}

\def\sectionautorefname{Section}
\def\subsectionautorefname{Subsection}


\begin{abstract}
 We have revised the softness property introduced by J\"org Brendle and Haim Judah (perfect sets of random reals. Israel J. Math., 83(1-2):153–176,18
1993), to present a new definition of a class of posets called $\sigma$-soft-linked. Our work demonstrates that these posets
 work well to preserve the evasion number as well as the bounding number small in generic extensions. Furthermore, we establish a connection between our concept and the Fr\'echet-linked notion introduced by Diego A. Mej\'ia (Matrix iterations with vertical support restrictions. In Proceedings of the 14th1
and 15th Asian Logic Conferences, pages 213–248. World Sci. Publ., Hackensack, NJ, 2019).
\end{abstract}
\maketitle

\makeatother

\section{Introduction}\label{sec:intro}

Blass's~\cite{blasse} study of set-theoretic aspects of the Specker phenomenon in Abelian group theory led him to introduce the concept of evasion and prediction, which is a combinatorial concept. He investigated several evasion numbers and how large they are compared to other cardinal characteristics. In particular, with the cardinals in Cicho\'n's diagram. Brendle~\cite{BrendlevasionI} extended these ideas and studied them more closely.

Fix $b\in((\omega+1)\smallsetminus2)^\omega$. An \emph{$\prod b$-predictor} is a pair $\pi=(D,\seq{\pi_n}{n\in D})$ such that $D\in[\omega]^{\aleph_0}$ and $\pi_n\colon\prod_{k<n}b(k)\to b(n)$ for $n\in D$. Let $\Pi_b$ be the collection of $\prod b$-predictors. 
 
For $\pi\in\Pi_b$ and $f\in\prod b$ write 
\[\pi\sqsubset^{\textbf{pr}} f\text{\ iff \ } f(n)=\pi_n(f{\upharpoonright}n)\text{\ for all but finitely many\ } n\in D\]
in which case we say that $\pi$ predict $f$; otherwise \emph{$f$ evades $\pi$}.  We define the corresponding~\emph{evasion number}
\[\efrak_b:=\efrak_{\prod b}=\min\set{|F|}{F\subseteq\prod b\textrm{\ and\ }\forall \pi\in\Pi_b\,\exists f\in F\,(f \text{\ evades\ } \pi)}.\]
It is clear that $b\leq c$ we have that $\efrak_c\leq\efrak_b$. In the case that $b$ eventually equals $\omega$, we get $\efrak:=\efrak_b$ the \emph{evasion number}; in case that $\prod b=n^\omega$ ($n\geq2$), we set $\efrak_n:=\efrak_b$. Finally,  we let $\eubd:=\min\set{\efrak_b}{b\in\baire}$ denote the \emph{unbounded evasion number} and let $\efrak_\textbf{fin}:=\min\set{\efrak_n}{n\in\omega}$ the \emph{finite evasion number}. It is well-known that $\efrak\leq\eubd\leq\efin$.  Below we summarize some ZFC results related to the evasion numbers. 

\begin{lemma}[{\cite[Sec.~4.2]{BrendlevasionI}}]\label{sumeva}\

\begin{enumerate}[label=\rm(\arabic*)]
    \item\label{sumevaa} $\efrak\geq\min\{\bfrak,\eubd\}$. Moreover, $\min\{\bfrak,\efrak\}=\min\{\bfrak,\eubd\}$, so $\bfrak>\efrak$ implies $\efrak=\eubd$.\smallskip

    \item $\efin=\efrak_n$ for all $n\geq2$.\smallskip

    \item $\efin\geq\sfrak$, where $\sfrak$ donotes the~\emph{splitting number}.\smallskip

    \item $\efrak\leq\cov(\Mwf)\leq\dfrak$. \smallskip

    \item $\efin\leq\non(\Ewf)$.\smallskip

    \item $\add(\Nwf)\leq\efrak$ and $\min\{\efrak,\bfrak\}\leq\add(\Mwf)$.
\end{enumerate}   
\end{lemma}
Here, given an ideal $\Iwf\subseteq\Pwf(X)$ containing all singletons, let
\begin{align*}
 \add(\Iwf)&=\min\set{|\Jwf|}{\Jwf\subseteq\Iwf,\,\bigcup\Jwf\notin\Iwf};\\
 \cov(\Iwf)&=\min\set{|\Jwf|}{\Jwf\subseteq\Iwf,\,\bigcup\Jwf=X};\\
 \non(\Iwf)&=\min\set{|A|}{A\subseteq X,\,A\notin\Iwf};\\
 \cof(\Iwf)&=\min\set{|\Jwf|}{\Jwf\subseteq\Iwf,\ \forall\, A\in\Iwf\ \exists\, B\in \Jwf\,(A\subseteq B)}.
\end{align*}
$\dfrak$ denotes, as usual the \emph{dominating number}, that is, the size of the
smallest family of functions $D\subseteq\baire$ such that for all $f\in\baire$ there is $g\in D$ with
$g(n)\geq f(n)$ for almost all $n$. In future, we shall abbreviate \emph{for almost all $n$} by $\forall^\infty n$, and write $f\leq^* g$ for $\forall^\infty n\, (f(n)\leq g(n))$. Call a family
of functions $F\subseteq\baire$ unbounded if for all $f\in\baire$  there is $g\in F$ which lies infinitely often above $f$. Then
\[\bfrak=\min\set{F\subseteq\baire}{F \text{\ is unbounded}}\]
is the \emph{unbounding number} which is dual to $\dfrak$. Finally, denote by $\Nwf$ and $\Mwf$ the $\sigma$-ideals of Lebesgue measure sets and of meager sets of~$\cantor$, respectively, and let $\Ewf$ be the ideal generated by $F_\sigma$ measure zero subsets of $\cantor$.
\newcommand{\NAwf}{\Nwf\!\Awf}

We complete this brief overview of ZFC results with another result of Brendle concerning the relationship between $\eubd$ and $\non(\NAwf)$ where $\NAwf$ denotes the $\sigma$-ideal of the null-additive subsets of $\cantor$ (see~\autoref{Def:nullad}). Namely, he states $\non(\NAwf)\leq\eubd$, but the proof
does not appear anywhere. For completeness, we offer a proof of this inequality in~\autoref{zfc-resul}  

As to consistency results, in~\cite{BrJ}, Brendle and Judad presented the property that helps us force $\bfrak$ small. Given a partial order $\Por$, a function $h\colon\Por\to\omega$ is a \emph{height function} iff $q\leq p$ implies $h(q)\geq h(p)$. A pair $\la\Por,h\ra$ has the property~\eqref{BrJprop} iff:
\begin{equation}
 \parbox{0.8\textwidth}{
 given a maximal antichain $\set{p_n}{n\in\omega}\subseteq\Por$ and $m\in\omega$, there is an $n\in\omega$ such that: whenever $p$ is incompatible with $\set{p_j}{j\in n}$ then $h(p)>m$.  
}
\tag{\faPagelines}
\label{BrJprop}
\end{equation}
They established that FS (finite support) iterations of ccc posets that fulfill~(\ref{BrJprop}) do not add dominating reals (so they do not increase $\bfrak$ small). Later, Brendle~\cite{BrendlevasionI} enhanced~\eqref{BrJprop} to prove the consistency of $\efrak<\eubd$. He showed that FS iterations of ccc posets with the property~\eqref{Breva} do not add predictors (so they do not increase $\efrak$) nor add dominating reals. We say that a pair $\la\Por,h\ra$ has the property~\eqref{Breva} iff:
\begin{equation}
 \parbox{0.8\textwidth}{
\begin{enumerate}[label=\rm (\Roman*)]
    \item given a $p\in\Por$, a maximal antichain $\set{p_n}{n\in\omega}\subseteq\Por$ of conditions below $p$ and $m\in\omega$, there is an $n\in\omega$ such that: whenever $q\leq p$ is incompatible with $\set{p_j}{j\in n}$ then $h(p)>m$; and \smallskip
    
    \item if $p, q\in\Por$ are compatible, then there is $r\leq p,q$ so that $h(r)\leq h(p)+h(q)$.
\end{enumerate}
}
\tag{\faThemeisle}
\label{Breva}
\end{equation}

Drawing inspiration from the property labeled as~(\ref{Breva}), the objective of this note is to establish a class of partially ordered sets (posets) referred to as \emph{$\sigma$-soft-linked}~(\autoref{maindef}). This class encompasses widely recognized types of forcing notions such as eventually different real forcing, meager forcing,  Dirichlet–Minkowski forcing, and others. Furthermore, we will provide evidence that any poset in this category, which is definable, maintains a small value of $\efrak$. Additionally, we will demonstrate that our notion is connected to the concept of ``Fr\'echet-linked in $\Por$" (abbreviated as $\Fr$-linked) as introduced by Mej\'ia~\cite{mejvert}.

\noindent\textbf{Notation.} A~\emph{forcing notion} is a pair $\la\Por,\leq\ra$ where $\Por\neq\emptyset$ and $\leq$ is a relation on~$\Por$ that satisfies reflexivity and 
transitivity. We also use the expression pre-ordered set (abbreviated p.o.\ set or just poset) to refer to a forcing notion. The elements of $\Por$ are called conditions and we say that a condition $q$ is \emph{stronger} than a condition $p$ if $q\leq p$.

\begin{definition}
Let $\Por$ be a forcing notion. 
\begin{enumerate}
    \item Say that $p, q\in\Por$ are \emph{compatible} (in $\Por$), denoted by $p\,\|_\Por\,q$, if $\exists r\in\Por\,(r\leq p\textrm{\ and\ }r\leq q)$. Say that $p, q\in\Por$ are \emph{incompatible} (in $\Por$) if they are not compatible in $\Por$, which is
denoted by $p\perp_\Por q$. 

When $\Por$ is clear from the context, we just write $p\,\|\,q$ and $p\perp q$.\smallskip

    \item Say that $A\subseteq\Por$ is an \emph{antichain} if $\forall p, q\in\Por\,(p\neq q\Rightarrow p\perp q)$. $A$ is a \emph{maximal antichain} on $\Por$ iff $A$ is an antichain and $\forall p\in\Por\,\exists q\in A\,(p\,\|_\Por\,q)$.\smallskip
    
    \item Say that $D\subseteq\Por$ is \emph{dense} (in $\Por$) if $\forall p\in\Por\,\exists q\in D\,(q\leq p)$.\smallskip

    \item Say that $G\subseteq\Por$ is a \emph{$\Por$-filter} if it satisfies
    \begin{enumerate}
        \item $G\neq\emptyset$;

        \item for all $p, q\in G$ there is some $r\in G$ such that $r\leq p$ and $r\leq q$; and 

        \item if $p\in\Por$, $q\in G$ and $q\leq p$, then $p\in G$.\smallskip
    \end{enumerate}

    \item Let $\mathcal D$ be a family of dense subsets of $\Por$. Say that $G\subseteq\Por$ is \emph{$\Por$-generic over $V$} if $G$ is
a $\Por$-filter and $\forall D\in\mathcal D\,(G\cap D\neq\emptyset)$. Denote by $\dot G_\Por$ the canonical name of the generic set. When $\Por$ is clear from the context, we just
write  $\dot G$.
\end{enumerate}
\end{definition}

\begin{fact}[{\cite{Gold}}]\label{basfor}
  Let $\Por$ be a forcing notion. Let $p, q\in\Por$. 
   \begin{enumerate}[label=\rm(\arabic*)]
\item $p\perp q$ iff $q\Vdash p\notin\dot G$.\smallskip

\item $G\subseteq\Por$ is a $\Por$-generic over $V$ iff for every maximal antichain $A\in V$, $|G\cap A|=1$.
   \end{enumerate}
\end{fact}

Define the order $\leq^*$ in $\Por$ as $q\leq^*p$ iff $\forall r\leq q\,(r\parallel p)$, i.e, $p\leq^* q\Leftrightarrow\{q\}$~is predense below~$p$. $\Pbb$~is separative if and only if ${\leq^*}$ equals~${\le}$. For $p\in\Pbb$ let $\Pbb|p=\set{q\in\Pbb}{q\le p}$.

We now review some basic notation about relational systems. A~\emph{relational system} is a triple~$\Rbf=\la X, Y, R\ra$ where $R$ is a relation and $X$ and $Y$ are non-empty sets. Such a relational system has two cardinal characteristics associated with it:
\begin{align*}
    \bfrak(\Rbf)&:=\min\set{|F|}{F\subseteq X\text{\ and }\neg\exists y\in Y\, \forall x\in F\colon x R y},\\
    \dfrak(\Rbf)&:=\min\set{|D|}{D\subseteq Y\text{\ and }\forall x\in X\, \exists y\in D\colon x R y}.
\end{align*}
Given another relational system $\Rbf'=\la X',Y',R'\ra$,  say that a pair $(\Psi_-,\Psi_+)\colon\Rbf\to\Rbf'$ is a \emph{Tukey connection from $\Rbf$ into $\Rbf'$} if 
 $\Psi_-\colon X\to X'$ and $\Psi_+\colon Y'\to Y$ are functions such that  $\forall\, x\in X\ \forall\, y'\in Y'\colon \Psi_-(x) R' y'$ then $ x R \Psi_+(y')$. Say that $\Rbf$ is \emph{Tukey below} $\Rbf'$, denoted by $\Rbf\leqT\Rbf'$, if there is a Tukey connection from $\Rbf$ to $\Rbf'$. 
 Say that $\Rbf$ is \emph{Tukey equivalent} to $\Rbf'$, denoted by $\Rbf\eqT\Rbf'$, if $\Rbf\leqT\Rbf'$ and $\Rbf'\leqT\Rbf$. It is well-known that $\Rbf\leqT\Rbf'$ implies $\bfrak(\Rbf')\leq\bfrak(\Rbf)$ and $\dfrak(\Rbf)\leq\dfrak(\Rbf')$. Hence, $\Rbf\eqT\Rbf'$ implies $\bfrak(\Rbf')=\bfrak(\Rbf)$ and $\dfrak(\Rbf)=\dfrak(\Rbf')$.

For instance, the cardinal characteristics associated with an ideal can be characterized by relational systems.

\begin{example}\label{exm:Iwf}
For $\Iwf\subseteq\pts(X)$, define the relational systems: 
\begin{enumerate}[label=(\arabic*)]
    \item $\Iwf:=\la\Iwf,\Iwf,\subseteq\ra$, which is a directed preorder when $\Iwf$ is closed under unions (e.g.\ an ideal).
    
    \item $\Cbf_\Iwf:=\la X,\Iwf,\in\ra$.
\end{enumerate}
It is well-known that, whenever $\Iwf$ is an ideal on $X$ containing $[X]^{<\aleph_0}$,
\begin{multicols}{2}
\begin{enumerate}[label= \rm (\alph*)]
    \item $\bfrak(\Iwf)=\add(\Iwf)$. 
    
    \item $\dfrak(\Iwf)=\cof(\Iwf)$.
    
    \item $\bfrak(\Cbf_\Iwf)=\non(\Iwf)$.

    \item $\dfrak(\Cbf_\Iwf)=\cov(\Iwf)$. 
\end{enumerate}
\end{multicols}
\end{example}

Here, as usual, given a formula $\phi$, $\exists^\infty\, n<\omega\colon \phi$ means that infinitely many natural numbers satisfy $\phi$.

Lastly, for $x\in\Ror$ let $\parallel x\parallel$ denote the distance of $x$ to the nearest integer.

\section{Evasion number and localization}\label{zfc-resul}

This section aims to prove the following theorem:

\begin{theorem}[{\cite{BrendlevasionI}}]\label{evanna}
$\non(\NAwf)\leq\eubd$.
\end{theorem}

On the other hand, recently, the author, along with Diego Mej\'ia and Ismael Rivera-Madrid~\cite{CMR2} proved that $\add(\NAwf)=\non(\NAwf)$, so we obtain $\add(\NAwf)\leq\eubd$. Before entering into the details of the proof, we review some basic notions and notation:

\begin{definition}\label{Def:nullad}
A set $X\subseteq2^\omega$ is termed \emph{$\Nwf$-additive} if,  $A+X\in\Nwf$ for every $A\in\Nwf$. Denote by $\NAwf$ the collection of the $\Nwf$-additive subsets of $\cantor$. Notice that $\NAwf$ is a $\sigma$-ideal and $\NAwf\subseteq\Nwf$.
\end{definition}

We now introduce some terminology associated with slaloms. 

\begin{definition}\label{defloc}
Given a sequence of non-empty sets $b = \seq{b(n)}{n\in\omega}$ and $h\colon \omega\to\omega$, define 
\begin{align*}
 \prod b &:= \prod_{n\in\omega}b(n),  \\
 \Swf(b,h) &:= \prod_{n\in\omega} [b(n)]^{\leq h(n)}.
\end{align*}
For two functions $x\in\prod b$, $\varphi\in\Swf(b,h)$ and $w\in[\omega]^{\aleph_0}$, we define

\begin{enumerate}
    \item $x\,\in^*\varphi\textrm{\ iff\ }\forall^\infty n\in\omega\,(x(n)\in \varphi(n))$,\smallskip

    \item $x\in^{\circ}(\varphi,w)$ iff $\forall^\infty n\in w\,(x(n)\in \varphi(n))$.
\end{enumerate}
We define the \emph{localization cardinal} 
\[\blc_{b,h}=\min\set{|F|}{F\subseteq \prod b\;\&\;\neg\exists \varphi \in \Swf(b,h)\,\forall x \in F\,(x\in^* \varphi)},\]
and a \emph{variant of the localization cardinal}
\[\blcp_{b,h}=\min\set{|F|}{F\subseteq \prod b\;\&\;\neg\exists (\varphi,w) \in \Swf(b,h)\times[\omega]^{\aleph_0}\,\forall x \in F\,(x\in^{\circ}(\varphi,w))}.\]
It is not hard to see that $\blc_{b,h}\leq\blcp_{b,h}$ holds. Let  $\minLc:=\min\set{\blc_{b,\id_\omega}}{b\in\baire}$ where $\id_\omega$ denotes the identity function on $\omega$.
\end{definition}

Pawlikowski provided a characterization of the uniformity of $\NAwf$ in terms of localization cardinals. Namely, he proved:

\begin{theorem}[{\cite{paw85}}]\label{thm:paw}
 $\non(\NAwf)=\minLc$.    
\end{theorem}

On the other hand, it was proved that $\minLc=\min\set{\blc_{b,h}}{b\in\baire}$ when $h$ goes to infinity in \cite[Lemma~3.8]{CM}. Therefore, $\non(\NAwf)=\min\set{\blc_{b,h}}{b\in\baire}$. 


\newcommand{\dlcp}{\mathfrak{d}^{\mathrm{Lc_0}}}
 \begin{definition}
Let $b, h$ be as in~\autoref{defloc}.
\begin{enumerate}[label=\rm(\arabic*)]
    \item Define $\Lc(b,h):=\la\prod b,\Scal(b,h),\in^*\ra$ ($\Lc$ stands for \emph{localization}), which is a relational system. Notice that $\blc_{b,h}:=\bfrak(\Lc(b,h))$ and $\dlc_{b,h}:=\dfrak(\Lc(b,h))$.

    \item Define $\Lc_0(b,h):=\la\prod b,\Scal(b,h)\times[\omega]^{\aleph_0},\in^\circ\ra$, which is a relational system. Put $\blcp_{b,h}:=\bfrak(\Lc_0(b,h))$ and $\dlcp_{b,h}:=\dfrak(\Lc_0(b,h))$.

    \item Assume $b\in((\omega+1)\smallsetminus2)^\omega$. Define $\mathbf{E}_b=\la\prod b,\Pi_b,\sqsubset^{\textbf{pr}}\ra$ ($\mathbf{E}_b$ stands for \emph{unbounded evasion}), which is a relational system. So it is clear that $\bfrak(\mathbf{E}_b)=\efrak_b$ and $\dfrak(\mathbf{E}_b)=\mathfrak{pr}_b$ ($\mathfrak{pr}_b$ is dual to $\efrak_b$ in a natural sense). 
\end{enumerate}
\end{definition}

\begin{remark}\label{Rem:Rb}
 Notice that, for a fixed $\pi\in\Pi_b$, $\set{f\in\prod b}{f\sqsubset^{\textbf{pr}}\pi}$ is meager set whenever $b\in((\omega+1)\smallsetminus2)^\omega$. Consequently, $\Cbf_\Mwf\leqT\mathbf{E}_b$, which implies that $\bfrak(\mathbf{E}_b)\leq\non(\Mwf)$ and $\cov(\Mwf)\leq\dfrak(\mathbf{E}_b)$.
\end{remark}

\begin{proof}[Proof of~\autoref{evanna}]
In view of~\autoref{thm:paw}, it suffices to prove~$\minLc\leq\eubd$. 
Let $b\in\baire$. Define $c_b\in\baire$ by $c_b(n)=\prod_{i\in I_n}b(i)$ where $I_n:=[\frac{n(n+1)}{2},\frac{n(n+1)}{2}+n]$. Note that $\seq{I_n}{n\in\omega}$ is a partition of $\omega$ into adjacent intervals such that the $n$-th interval $I_n$ has size $n+1$. Hence, it is enough to prove that $\mathbf{E}_b\leqT\Lc_0(c_b,\id_\omega)$.  To see this, we have to find maps $\Psi_-\colon\prod b\to\prod c_b$ and $\Psi_+\colon\Swf(c_b,\id_\omega)\to\Pi_b$ such that, for any $x\in\prod b$ and for any $\varphi\in\Swf(c_b,\id_\omega)$, $\Psi_-(x)\in^*\varphi$ implies $x=^{\textbf{pr}}\Psi_+(\varphi)$.

Let $x\in F$. Define $y_x\in \prod c_b$ by $y_x(n)=x{\upharpoonright}I_n$ for each $n\in\omega$. Put $\Psi_-(x)=y_x$. On the other hand, for $\varphi\in\Swf(c_b,\id_\omega)$ and for $w\in[\omega]^{\aleph_0}$. We can interpret $\varphi(n)$ as a~subset of the $(n+1)$-dimensional $Q$-vector space $Q^{I_n}$. Since $|\varphi(n)|=n$, there is a non-zero functional $a^n\colon Q^{I_n}\to Q$ which annihilates every vector in $\varphi(n)$. For $i\in I_n$ let $e_i\in Q^{I_n}$ be the unit vector defined by $e_i(j)=\delta_{i,j}$ (Kronecker), let $a_i^n=a^n(e_i)$ and let $d_n=\max\set{i\in I_n}{a^n_i\neq0}$. Then it is clear that for every $x\in F$ and $n\in w$ with $y_x(n)\in\varphi(n)$, rewriting $a^n(y_x(n))=0$, we obtain
 \[x(d_n)=-\frac{1}{d_n}\sum_{i\in I_n\cap d_n}a_i^n\cdot x(i).\]
 Now, define a predictor $\Psi_+(\varphi)=\seq{\pi_d}{d\in D}$ by 
 \[D:=\set{d_n}{n\in\omega}\textrm{\ and\ }\pi_d(\sigma)=-\frac{1}{d_n}\sum_{i\in I_n\cap d_n}a_i^n\cdot \sigma(I).\]
It is not hard to see that $y_x\in^*\varphi$ implies $x=^{\textbf{pr}}\pi$.
\end{proof}

\section{Soft-linkedness}

Within this section, our primary objective is to formalize the approach of Brendle and Judad in order to introduce the central concept of this study, called “$\mathrm{leaf}$-linked”. Additionally, we will investigate the correlation between $\mathrm{leaf}$-linked and $\Fr$-linked, as demonstrated in~\autoref{SoftFr}.

\subsection{Soft-linkedness}
\

We now proceed to acquaint our linkedness property.

\begin{definition}\label{maindef}
  Let $\Por$ be a forcing notion and $\theta$ be a cardinal number.
\begin{enumerate}[label=\rm (\arabic*)]
    \item A set $Q\subseteq\Por$ is \emph{leaf-linked} if, for any maximal antichain $\set{p_n}{n\in\omega}\subseteq\Por$, there is some $n\in\omega$ such that $\forall p \in Q\,\exists j<n\,(p\,||\,p_j).$\smallskip
    
    \item\label{maindefa}  A set $Q\subseteq\Por$ is \emph{$\mathit{leaf}^*$-linked} if, for any $p\in\Por$ and for any maximal antichain $\set{p_n}{n\in\omega}\subseteq\Por$ of conditions below $p$, there is some $n\in\omega$ such that $\forall q \in Q\,[q\leq p)
    \Rightarrow\exists j<n\,(q\,||\,p_j)]$.\smallskip
    
    \item\label{maindefb}  Say that $\Por$ is \emph{$\theta$-soft-linked} if there is a sequence $\seq{Q_\alpha}{\alpha<\theta}$ of subsets of $\Por$ fulfilling the following:  
    \begin{enumerate}[label=\rm (\alph*)]
        \item\label{maindefba}  $Q_\alpha$ is $\mathrm{leaf}$-linked,  

        \item\label{maindefbb} $\bigcup_{\alpha<\theta}Q_\alpha$ is dense in $\Por$, and

         \item\label{maindefbc} $\forall\alpha, \alpha'<\theta\,\exists\alpha^*<\theta\,\forall p\in Q_\alpha\,\forall q\in Q_{\alpha'}\,[p\,||q\Rightarrow\exists r\in Q_{\alpha^*}(r\leq p\wedge r\leq q)]$.
    \end{enumerate}
\end{enumerate}
\end{definition}

\begin{remark}
Let $Q\subseteq\Por$. 
\begin{enumerate}
    \item If $Q$ is $\mathrm{leaf}$-linked, then $Q$ is $\mathrm{leaf}^*$-linked.\smallskip   

    \item $Q$ is $\mathrm{leaf}^*$-linked iff for all $p\in\Por$, $Q|p:=\set{q\in Q}{q\leq p}$ is $\mathrm{leaf}$-linked.  
\end{enumerate}
\end{remark}


In the ensuing discussion, our main objective is to establish the correlation between our concept and the $\sigma$-Fr\'echet-linked notion introduced by Mej\'ia. He demonstrated that a $\sigma$-Fr-linked poset does not add dominating reals and that it preserves a specific type of maximal almost disjoint (mad) families. The findings in Mej\'ia's work were later enhanced in the study conducted by~Brendle, the author, and Mej\'ia~\cite{BCM}, which served as the motivation for constructing matrix iterations of ${<}\theta$-uf-linked posets. The purpose of these iterations was to enhance the separation of the left-hand side of Cicho\'n's diagram from the study conducted by Goldstern, Mej\'ia, and Shelah~\cite{GMS}, by incorporating the inequality $\cov(\Mwf)<\dfrak=\non(\Nwf)=\cfrak$.

We begin with some notation: 

\begin{itemize}
    \item Denote by $\Fr:=\set{\omega\menos a}{a\in[\omega]^{<\aleph_0}}$ the \emph{Fr\'echet filter}.

    \item A filter $F$ on $\omega$ is \emph{free} if $\Fr\subseteq F$. A set $x\subseteq\omega$ is \emph{$F$-positive} if it intersects every member of $F$. Denote by $F^+$ the family of $F$-positive sets. Note that $x\in\Fr^+$ iff $x$ is an infinite subset of $\omega$.
\end{itemize}

\begin{definition}[{\cite{mejvert, BCM}}]\label{Def:Fr}
    Let $\Por$ be a poset and $F$ a filter on $\omega$. A set $Q\subseteq \Por$ is \emph{$F$-linked} if, for any $\bar p=\seq{p_n}{n<\omega} \in Q^\omega$, there is some $q\in \Por$ such that \[q\Vdash\set{n\in\omega}{p_n\in\dot G}\in F^+.\]
    Observe that, in the case $F=\Fr$, the above equation is “$\set{n\in\omega}{p_n\in\dot G}$ is infinite”. 
    
    We say that $Q$ is \emph{uf-linked (ultrafilter-linked)} if it is $F$-linked for any filter $F$ on $\omega$ containing the \emph{Fr\'echet filter} $\Fr$.
    
    For an infinite cardinal $\theta$, $\Por$ is \emph{$\mu$-$F$-linked} if $\Por = \bigcup_{\alpha<\theta}Q_\alpha$ for some $F$-linked $Q_\alpha$ ($\alpha<\theta$). When these $Q_\alpha$ are uf-linked, we say that $\Por$ is \emph{$\theta$-uf-linked}.
\end{definition}

It is clear that any uf-linked set $Q\subseteq \Por$ is F-linked, and consequently also $\Fr$-linked. But for ccc poset we have:

\begin{lemma}[{\cite[Lem~5.5]{mejvert}}]
If $\Por$ is ccc then any subset of $\Por$ is uf-linked iff it is Fr-linked. 
\end{lemma}

One of the main results of this section is:

\begin{theorem}\label{SoftFr}
Assume that $\Por$ is ccc. If $\Por$ is $\theta$-$\mathrm{soft}$-linked poset, then it is $\theta$-$\Fr$-linked.
\end{theorem}

In order to prove the theorem above, all we have to do is to prove that any subset of a poset~$\Por$ is $\Fr$-linked, it is still leaf-linked. However, it is worth noting that we have much more:

\begin{lemma}\label{thm:a1}
Let $\Por$ be a ccc forcing notion and a set $Q\subseteq\Por$. The following statements are equivalent:
\begin{enumerate}[label=\rm (\arabic*)]
    \item\label{Fr} $Q$ is $\Fr$-linked.\smallskip
    
    \item\label{Soft} $Q$ is $\mathrm{leaf}$-linked.\smallskip

    \item\label{ConFr} for each $\Por$-name $\dot n$ of a natural number there is some $m<\omega$ such that $\forall p\in Q\,(p\not\Vdash m<\dot n)$. 
\end{enumerate}
\end{lemma}
\begin{proof}
\noindent$\ref{Fr}\Rightarrow\ref{Soft}$: Let $\set{p_n}{n\in\omega}$ be a maximal antichain in $\Por$. Towards a~contradiction let us assume $\forall m\,\exists q_m\in Q\,\forall n<m\,(q_m\perp p_n)$. By $\Fr$-linkedness, choose $q\in\Por$ such that $q\Vdash``\exists^\infty m\in\omega\,(q_m\in\dot G)"$. Since each $q_m$ fulfills $\forall n<m\,(q_m\perp p_n)$, $q\Vdash\exists^\infty m\forall n<m\,(p_n\not\in\dot G)$. Hence, $q\Vdash\forall n<\omega\,(p_n\not\in\dot G)$. On the other hand, since $\set{p_n}{n\in\omega}$ is a maximal antichain, by using~\autoref{basfor} $\Vdash \exists n\in\omega\,(p_n\in\dot G)$, which establishes a~contradiction.

\noindent$\ref{Soft}\Rightarrow\ref{ConFr}$: Assume that $\dot n$ is a $\Por$-name of a natural number and assume towards a contradiction that for each $m\in\omega$, there is some $q_m\in Q$ that forces $m<\dot n$. Let $\set{p_k}{k\in\omega}$ be a maximal antichain deciding the value $\dot n$ and let $n_k$ be such that $p_k\Vdash``\dot n=n_k"$. By~soft-linkedness, choose $m\in\omega$ such that $\forall p\in Q\,\exists k<m\,(p\parallel p_k)$. Next, let $k^\star$ be larger than $\max\set{n_k}{k<m}$. Then $q_{k^\star}\Vdash k^\star<\dot n$ and since $q_{k^\star}\in Q$, then there is $k<m$ such that $q_{k^\star}\parallel p_k$. Let  $r\leq q_{k^\star}, p_k$. Then $r\Vdash \dot n=k<k^\star<\dot n$, which is a contradiction. 

\noindent$\neg\ref{Fr}\Rightarrow\neg\ref{ConFr}$: Assume that there is a sequence $\seq{p_{n}}{n < \omega}$ in $Q$ such that $\Vdash``\set{n\in\omega}{p_n\in \dot G}$ finite". Then there is a $\Por$-name $\dot M$ in $\omega$ such that $\Vdash\set{n\in\omega}{p_n\in \dot G}\subseteq\dot M$. Towards a~contradiction suppose that $\exists m\,\forall q\in Q\,(q\not\Vdash m<\dot M)$. Choose $m\in\omega$ such that $\forall q\in Q\,(q\not\Vdash m<\dot M)$. For any $n\in\omega$, $p_n\not\Vdash m<\dot M $. Next, find 
$r\leq p_m$ such that $r\Vdash m\geq\dot M$, on the other hand, since  $r\Vdash p_m\in G$, $r\Vdash m<\dot M$. This is a~contradiction. 
\end{proof}

\begin{remark}
We, along with Mej\'ia, proved the direction \noindent$\ref{ConFr}\Rightarrow\ref{Fr}$ in~\autoref{thm:a1}. 
\end{remark}

Next, we proceed to derive several consequences stemming from~\autoref{maindef}. These ramifications hold significant relevance in substantiating the theorem about the preservation of the evasion number, as explored further in~\autoref{thm:preva}.

\begin{lemma}\label{solfch}
Let $\theta$ be a cardinal number. Consider the following properties of a~poset~$\Pbb$:
\begin{enumerate}[label=\rm (\arabic*)]
\item\label{solfch:A} $\Por$ is $\theta$-soft-linked.\smallskip
\item\label{solfch:B}
There is a~sequence $\seq{ Q_\alpha}{\alpha<\theta}$ of soft-linked
subsets of\ ~$\Pbb$ such that $Q=\bigcup_{\alpha<\theta}Q_\alpha$ is dense
in~$\Pbb$ and there are functions $h\colon Q\to\theta$ and $g^\star\colon\theta\times Q\to\theta$ such that
$\forall\alpha<\theta$ $h^{-1}\llbracket\{\alpha\}\rrbracket\subseteq Q_\alpha$ and
\begin{equation*}
\forall\alpha<\theta\,
\forall p\in Q_{\alpha}\,
\forall p'\in Q\,
[p\parallel p'\Rightarrow\exists r\in Q_{g^*(\alpha,p')}\,(
r\le p\text{ and }r\le p')].
\label{solfch:Beq}
\tag{\faPencilSquareO$_\mathrm{2}$}
\end{equation*}
\item\label{solfch:C}
There is a~sequence $\seq{Q_\alpha}{\alpha<\theta}$ of soft-linked
subsets of~$\Pbb$ such that $Q=\bigcup_{\alpha<\theta}Q_\alpha$ is dense
in~$\Pbb$ and there is $\bar g\colon Q\times Q\to\theta$ such that
\begin{equation*}
\forall p,p'\in Q\,
[p\parallel p'\Rightarrow\exists r\in Q_{\bar g(p,p')}\,(
r\le p\text{ and }r\le p')].
\tag{\faPencilSquareO$_\mathrm{3}$}
\end{equation*}
\end{enumerate}
Then $\textup{\ref{solfch:A}}\Rightarrow\textup{\ref{solfch:B}}\Rightarrow\textup{\ref{solfch:C}}$.
\end{lemma}
\begin{proof}
In all cases, we can define $h\colon Q\to\theta$ by
\[h(p)=\min\set{\alpha<\theta}{p\in Q_\alpha}\] for
$p\in Q$ because $Q=\bigcup_{\alpha<\theta}Q_\alpha$.

\noindent$\ref{solfch:A}\Rightarrow\ref{solfch:B}$. By~employing~\autoref{maindef}~\ref{maindefb}~\ref{maindefbc}, we can define a function  $g\colon\theta\times\theta\to\theta$ such that $g\colon(\alpha,\alpha')\mapsto\alpha^*$. Define $g^\star\colon\theta\times Q\to\theta$ by $g^\star(\alpha,p)=g(\alpha,h(p))$ for
$(\alpha,p)\in\theta\times Q$.
Assume that $p\in Q_\alpha$ and $p'\in Q$ and $p\parallel p'$.
Then $p'\in Q_{h(p')}$ and $g^\star(\alpha,p')=g(\alpha,h(p'))$.
Therefore by~\ref{solfch:A} there is $r\in Q_{g(\alpha,h(p'))}=Q_{g^\star(\alpha,p')}$ such that
$r\le p,p'$.

\noindent$\ref{solfch:B}\Rightarrow\ref{solfch:C}$.
Define $G\colon Q\times Q\to\theta$ by $\bar g(p,p')=g^\star(h(p),p')$ for
$(p,p')\in Q\times Q$.
Assume that $p,p'\in Q$ and $p\parallel p'$.
Then $p\in Q_{h(p)}$ and by~\ref{solfch:B} there is $r\in Q_{g^\star(h(p),p')}=Q_{\bar g(p,p')}$
such that $r\le p,p'$.
\end{proof}

The following lemma establishes two more consequences of $\theta$-soft-linked and gives a characterization of $\sigma$-soft-linked in the particular case $\theta=\omega$.

\begin{lemma}\label{solft:one}
Let $\theta$ be a cardinal number and let $\Pbb$ be a poset. Consider the following properties of $\Por$:
\begin{enumerate}[label=\rm (\arabic*)]
\item \label{solft:onea}
$\Pbb$ is $\theta$-soft-linked.\smallskip
\item\label{solft:oneb}
There is a~dense set $Q\subseteq\Pbb$ and there are 
functions $h\colon Q\to\theta$ and $g\colon\theta\times\theta\to\theta$ with the
following properties:
\begin{enumerate}[label=\rm (\roman*)]
\item
For every $\alpha<\theta$ the set\/ $\{p\in Q\mid h(p)=\alpha\}$ is leaf-linked.
\item
$\forall p,q\in Q$
$[p\parallel q\Rightarrow
\exists r\in Q$ $(r\le p,q$ and $h(r)\le g(h(p),h(q)))]$.
\end{enumerate}\smallskip
\item \label{solft:onec}
There is a~dense set $Q\subseteq\Pbb$ and there are
functions $h\colon Q\to\theta$ and $g\colon\theta\times\theta\to\theta$ with the
following properties:
\begin{enumerate}[label=\rm (\roman*)]
\item
For every $\alpha<\theta$ the set\/ $\{p\in Q\mid h(p)\le\alpha\}$ is
leaf-linked.
\item
$\forall p,q\in Q$
$[p\parallel q\Rightarrow
\exists r\in Q$ $(r\le p,q$ and $h(r)\le g(h(p),h(q)))]$.
\end{enumerate}
\end{enumerate}
Then $\ref{solft:onec}\Rightarrow\ref{solft:oneb}$; $\ref{solft:onea}\Rightarrow\ref{solft:oneb}$; $\ref{solft:onec}\Rightarrow\ref{solft:onea}$ for regular~$\theta$; and
$\ref{solft:onea}\Leftrightarrow\ref{solft:oneb}\Leftrightarrow\ref{solft:onec}$ for $\theta=\omega$.
\end{lemma}

\begin{remark}
The function $h$ in conditions~\ref{solft:oneb} and~\ref{solft:onec} of~\autoref{solft:one} may be not monotone.  
\end{remark}

\begin{proof}[Proof of~\autoref{solft:one}]
The implication \indent$\ref{solft:onec}\Rightarrow\ref{solft:oneb}$ is trivial.

\noindent$\ref{solft:onea}\Rightarrow\ref{solft:oneb}$. Let $\langle Q_\alpha\mid\alpha<\theta\rangle$ be a~sequence of 
subsets of~$\Pbb$ witnessing that $\Pbb$~is $\theta$-soft-linked.
The set $Q:=\bigcup_{\alpha<\theta}Q_\alpha$ is a~dense subset of~$\Pbb$.
For $p\in Q$ define
$h(p)=\min\{\alpha<\theta\mid p\in Q_\alpha\}$.
By using~\autoref{maindef}~\ref{maindefb}~\ref{maindefbc} for every $\alpha,\beta<\theta$ define
\[g(\alpha,\beta)=\min\set{\gamma<\theta}{\forall p\in Q_\alpha
\forall q\in Q_\beta\,[p\parallel q\Rightarrow\exists r\in Q_\gamma\,(r\le p,q)]}.\]
Condition (i)~holds because every set $\{p\in Q\mid h(p)=\alpha\}$ is a~subset
of the leaf-linked set~$Q_\alpha$ and is therefore leaf-linked.
Assume $p,q\in Q$ and $p\parallel q$.
Then $p\in Q_{h(p)}$ and $q\in Q_{h(q)}$ and so there is $r\in Q_{g(h(p),h(q))}$
such that $r\le p$ and $r\le q$.
By definition of~$h$, $h(r)\le g(h(p),h(q))$ and hence (ii)~holds.

\noindent$\ref{solft:onec}\Rightarrow\ref{solft:onea}$ for regular~$\theta$.
By~using~(i) for every $\alpha<\theta$ the set $Q_\alpha=\set{p\in Q}{h(p)\le\alpha}$
is leaf-linked and $Q=\bigcup_{\alpha<\theta}Q_\alpha$ is a~dense subset
of~$\Pbb$.
We verify the condition \ref{maindefbc} in ~\autoref{maindef}~\ref{maindefb}.
Fix $\alpha,\alpha'\le\theta$.
Define \[\alpha^*=\sup\set{g(\xi,\eta)}{\xi,\eta\le\max\{\alpha,\alpha'\}};\]
$\alpha^*<\theta$ because $\theta$~is regular.
Assume $p\in Q_\alpha$, $q\in Q_{\alpha'}$, and $p\parallel q$.
Then $g(h(p),h(q))\le\alpha^*$ because $h(p)\le\alpha$ and $h(q)\le\alpha'$.
By~using~(ii) there is $r\le p,q$ such that $h(r)\le g(h(p),h(q))\le\alpha^*$ and
hence $r\in Q_{\alpha^*}$.\smallskip

\noindent$\ref{solft:oneb}\Rightarrow\ref{solft:onec}$ for $\theta=\omega$.
Note that $\set{p\in Q}{h(p)\le n}$ is a~finite union of leaf-linked sets
$\set{p\in Q}{h(p)=k}$ for $k\le n$ and is therefore leaf-linked.
\end{proof}
\newcommand{\up}{up}
\newcommand{\mup}{$\uparrow$up}


In the rest of this section, we study some properties of pairs $\la \Por, h\ra$ which allows us to guarantee that  $\Por$ is  $\sigma$-soft-linked. We next fix some terminology. 

Let $P, Q\subseteq\Por$ and $p\in\Por$. We write $p\perp P$ if $p\perp q$ for all $q\in P$ and $Q\perp P$ if $p\perp P$ for every $p\in Q$.
We consider the following property of pairs $\la\Pbb,h\ra$ where $\Por$ is a poset and and $h$ a height function on $\Por$: Say that $\la\Pbb,h\ra$ has \emph{flash property} if:
\begin{equation}
 \parbox{0.8\textwidth}{
\begin{enumerate}[label=\rm (\Roman*)]
    \item given $m\in\omega$ and a 
 finite non-maximal non-empty antichain $P\subseteq\Por$, there is some finite subset $R$ of 
$\Por$ such that 
$R\perp P$ and $\forall p\in\Pbb$
$[(p\perp P$ and $h(p)\le m)\Rightarrow\exists r\in R\, (p\le r)]$; and \medskip
    
    \item if $p, q\in\Por$ are compatible, then there is $r\leq p,q$ so that $h(r)\leq h(p)+h(q)$.
\end{enumerate}
}
\tag{\faFlash}
\label{Miropr}
\end{equation}

\begin{lemma}
A~ccc poset $\Por$ with a~height function $h\colon\Por\to\omega$
satisfying~\emph{(\ref{Miropr})} is $\sigma$-soft-linked.
\end{lemma}
\begin{proof}
For $m\in\omega$ let \[Q_m:=\set{p\in\Por}{h(p)\leq m}\textrm{\ and\ } Q:=\bigcup_{m\in\omega}Q_m.\] 
To prove the lemma, it suffices to verify the property~\ref{solft:onec} in~\autoref{solft:one}
for $\theta=\omega$, for the height function $h$, and for the function
$g(n,m)=n+m$.
Condition~(I) is fulfilled by property~(II).
To prove the condition~(i) we must prove that every set $Q_m$ is leaf-linked. Let $P=\set{p_k}{k\in\omega}$ be arbitrary maximal antichain in~$\Pbb$.
Denote $P_n=\set{p_k}{k\le n}$.
By (I) of~(\ref{Miropr}), for every set~$P_n$ there is some finite subset $R_n$ of $\Por$ such that
$R_n\perp P_n$ and $\forall p\in Q_m\,[p\perp P_n\Rightarrow\exists r\in R_n\,(p\le r)]$.
We can assume that $R_n\in[Q_m]^{<\omega}$ because
$(p\in Q_m$ and $p\le r)\Rightarrow r\in Q_m$.
Therefore,
\begin{equation*}
\forall n\in\omega\, \exists R_n\in[Q_m]^{<\omega}\smallsetminus\,\{\emptyset\}\,
[R_n\perp P_n\text{ and }
\forall p\in Q_m\, (p\perp P_n\Rightarrow\exists r\in R_n\, (r\le p))].
\tag{$\oplus$}
\label{ch:soft}
\end{equation*}
By~\eqref{ch:soft} for every $p\in R_{n+1}$ there is $r\in R_n$ such that $r\le p$
(because $p\in Q_m$ and $p\perp P_n$).
In this way, by induction, we prove that
$\forall n\in\omega\,\forall p\in R_n\,\exists r\in R_0\,(r\le p)$.
Denote by $R_n'$ the set $\set{r\in R_0}{\exists p\in R_n\,(r\le p)}$.
Then~(\ref{ch:soft}) holds with the sets $R_n'$ instead of~$R_n$
and therefore without loss of generality, we can assume that $R_n\subseteq R_0$
for all $n\in\omega$.
Since $P$~is a~maximal antichain, there is $m_0\in\omega$
such that $\forall p\in R_0\,(p\not\perp P_{m_0})$. Then it follows that $R_{m_0}=\emptyset$ because $R_{m_0}\perp P_{m_0}$ and
$R_{m_0}\subseteq R_0$.
Then by~(\ref{ch:soft}) for $n=m_0$ we get $\forall p\in Q_m\,(p\not\perp P_{m_0})$.
This finishes the proof that $Q_m$~is leaf-linked.
\end{proof}

\begin{remark}\label{var}
Consider the following variants of (I) of~(\ref{Miropr}):
\begin{enumerate}[label=\rm (I$^\arabic*$)]
\item\label{var1}
$\forall m\in\omega\,\exists R\in[\Pbb]^{<\omega}\,\forall p\in\Pbb\,[h(p)\le m\Rightarrow\exists r\in R\,(p\le r)]$.\smallskip
\item\label{var2}
$\forall m\in\omega\,\forall q\in\Pbb\smallsetminus\{1_\Pbb\}\,\exists R\in[\Pbb]^{<\omega}\,(R\perp q)$ and $\forall p\in\Pbb\,[(p\perp q$ and $h(p)\le m)\Rightarrow\exists r\in R\,(p\le r)]$.\smallskip
\item\label{var3}
$\forall m\in\omega\,(\forall q',q\in\Pbb$ with $q\le q'$ and $q'\not\leq^* q)\,\exists R\in[\Pbb]^{<\omega}\,(R\le q'\textrm{\ and\ }R\perp q$) and $\forall p\in\Pbb\,[(p\le q'$ and $p\perp q$ and $h(p)\le m)\Rightarrow\exists r\in R\,(p\le r)]$.
\end{enumerate}
Note that~\ref{var1} and~\ref{var2} are instances
of~(I) for $P=\emptyset$ and $|P|=1$, respectively.
One can observe that
\begin{itemize}
\item
$\text{(I)}\Rightarrow$~\ref{var2};\smallskip
\item \ref{var3}${}\Leftrightarrow\forall q'\in\Pbb$ $(\Pbb|q'$
satisfies~\ref{var2}); \smallskip
\item
if $\Pbb$ has the largest element, then
\ref{var3}\/ $\Rightarrow$~\ref{var2}.
\end{itemize}
\end{remark}

We also consider the following property of pairs $\la\Pbb,h\ra$ where $\Por$ is a poset and and~$h$ a height function on $\Por$: Say that $\la\Pbb,h\ra$ has the \emph{club property} if:
\begin{equation}
 \parbox{0.8\textwidth}{
\begin{enumerate}[label=\rm(\Roman*)]

   \item if $\set{p_n}{n<\omega}$ is decreasing and $\exists m\in\omega\,\forall n\in\omega\,(h(p_n)\leq m)$, then $\exists p\in\Por\,\forall n\in\omega\,(p\leq p_n)$;\smallskip

     \item $\forall m\in\omega\,\forall P\in[\Por]^{<\omega}\smallsetminus\{\emptyset\}\,\exists R\in[\Por]^{<\omega}\,(R\perp P)$ and $\forall p\in\Por\,(h(p)\leq m)\,[p\perp P\Rightarrow\exists r\in R\,(p\le r)]$;\smallskip

    \item if $p, q\in\Por$ are compatible, then there is $r\leq p,q$ so that $h(r)\leq h(p)+h(q)$.
   
\end{enumerate}
 }
 \label{Brtrevol}
 \tag{$\clubsuit$}
 \end{equation}

The proof of the coming lemma relies essentially on ~\cite[Lem.~1.2]{BrJ}.

\begin{lemma}\label{Br:lemma}
A~ccc poset $\Por$ with a~height function $h\colon\Por\to\omega$
satisfying~\emph{(\ref{Brtrevol})} is $\sigma$-soft-linked.
\end{lemma}
\begin{proof}
For $m\in\omega$ let \[Q_m:=\set{p\in\Por}{h(p)\leq m}\textrm{\ and\ } Q:=\bigcup_{m\in\omega}Q_m.\] 
We verify \ref{up-softc} of~\autoref{solft:one} for $\theta=\omega$, for the height function $h$, and for the function $g(n,m)=n+m$. The only nontrivial condition is (i), so we show it, i.e., that each $Q_m$ is leaf-linked. To this, assume that $\set{p_n}{n\in\omega}\subseteq\Por$ is a  maximal antichain in $\Por$ and show that there is some $n\in\omega$ such that $\forall q \in Q_m\,\exists j<n\,(q\,||\,p_j)$. Suppose not. Then 
for each $n\in\omega$ let $R_n$ be a finite subset of $\Por$ obtained by employing~(II) of~(\ref{Brtrevol}) to the family $P=\set{p_i}{i<n}$. For each $n\in\omega$, enumerate $R_n$ as $\set{q_j^n}{j<k_n}$ for $k_n<\omega$. 

By assumption, none of these sets can be empty and we can assume that $q_j^n\in Q_m$ for all $j$,~$n$. By~applying~(II) of~(\ref{Brtrevol}) they form
an $\omega$-tree with finite levels with respect to $``\leq"$. Then by Konig's lemma, there is a function $f\in\omega^\omega$ such that $q^0_{f(0)}\leq q^1_{f(1)}\leq q^3_{f(3)}\leq\cdots.$ By~(I) of~(\ref{Brtrevol}) there is a condition $q\leq q_{f(n)}^n$ for all $n$, contradicting
the fact that $\set{p_n}{n\in\omega}$ is a maximal antichain.   
\end{proof}

\subsection{Weaknnes versions of Soft-linkedness}
\

Continuing our study of soft-linkedness, we collect now two new weak notions concerning not adding predictors. The following definitions are related to~\autoref{maindef} and point in this direction.

\begin{definition}\label{def:wsolf}
 Let $\Por$ be a forcing notion and $\theta$ be a cardinal number.
\begin{enumerate}[label=\rm (\arabic*)]
    \item\label{def:upsoft} We say that \emph{$\Pbb$ is $\theta$-\up-soft-linked} if there is a~sequence
$\seq{Q_\alpha}{\alpha<\theta}$ of upwards closed leaf-linked subsets
of~$\Pbb$ such that\smallskip
\begin{enumerate}[label=\rm (\alph*)]
    \item\label{def:upsofta} $\bigcup_{\alpha<\theta}Q_\alpha$ is dense in~$\Pbb$, and \smallskip

    \item\label{def:upsoftb} \autoref{maindef}~\ref{maindefb}~\ref{maindefbc} holds.\smallskip
    
\end{enumerate}

    \item\label{def:rupsoft} We say that \emph{$\Pbb$ is $\theta$-\mup-soft-linked} if there is an increasing
sequence $\seq{ Q_\alpha}{\alpha<\theta}$ of upwards closed leaf-linked
subsets of~$\Pbb$ such that\smallskip
\begin{enumerate}[label=\rm (\alph*)]
    \item\label{def:rupsofta} $\bigcup_{\alpha<\theta}Q_\alpha$ is dense in~$\Pbb$ and \smallskip

    \item\label{def:rupsoftb} \autoref{maindef}~\ref{maindefb}~\ref{maindefbc} holds.
\end{enumerate}
\end{enumerate} 
Notice that $\theta$-soft-linked\/ $\Rightarrow\theta$-\up-soft-linked\/   $\Rightarrow\theta$-\mup-soft-linked.
\end{definition}

Just as in~\autoref{solft:one}, we have:

\begin{lemma}\label{up-soft}
For a~poset\/~$\Pbb$ and a cardinal number $\theta$ consider the following properties:
\begin{enumerate}[label=\rm (\arabic*)]
\item\label{up-softa}
$\Pbb$ is $\theta$-\up-soft-linked.\smallskip
\item\label{up-softb}
There is a~dense set $Q\subseteq\Pbb$ and there are
functions $h\colon Q\to\theta$ and $g\colon\theta\times\theta\to\theta$ with the
following properties:
\begin{enumerate}[label=\rm (\roman*)]
\item
For every $\alpha<\theta$ the set\/ $\{p\in Q\mid h(p)=\alpha\}$ is leaf-linked.
\item
$\forall p,q\in Q$
$[p\parallel q\Rightarrow
\exists r\in Q$ $(r\le p,q$ and $h(r)\le g(h(p),h(q)))]$.
\item
$\forall p,q\in Q$ $q\le p\Rightarrow h(q)\ge h(p)$.\smallskip
\end{enumerate}
\item\label{up-softc}
There is a~dense set $Q\subseteq\Pbb$ and there are
functions $h\colon Q\to\theta$ and $g\colon\theta\times\theta\to\theta$ with the
following properties:
\begin{enumerate}[label=\rm (\roman*)]
\item
For every $\alpha<\theta$ the set\/ $\{p\in Q\mid h(p)\le\alpha\}$ is
leaf-linked.
\item
$\forall p,q\in Q$
$[p\parallel q\Rightarrow
\exists r\in Q$ $(r\le p,q$ and $h(r)\le g(h(p),h(q)))]$.
\item
$\forall p,q\in Q\,(q\le p\Rightarrow h(q)\ge h(p))$.
\end{enumerate}
\end{enumerate}
Then $\ref{up-softc}\Rightarrow\ref{up-softb}$ and $\ref{up-softa}\Rightarrow\ref{up-softb}$;
$\ref{up-softc}\Rightarrow\ref{up-softa}$ for regular~$\theta$; and
$\ref{up-softa}\Leftrightarrow\ref{up-softb}\Leftrightarrow\ref{up-softc}$ for $\theta=\omega$.
\end{lemma}
\begin{proof}
$\ref{up-softc}\Rightarrow\ref{up-softb}$ is trivial.

\noindent$\ref{up-softa}\Rightarrow\ref{up-softb}$
Let $\langle Q_\alpha\mid\alpha<\theta\rangle$ be a~sequence of upwards closed
 subsets of~$\Pbb$ witnessing that $\Pbb$~is $\theta$-\up-soft-linked.
The set $Q=\bigcup_{\alpha<\theta}Q_\alpha$ is a~dense subset of~$\Pbb$.
For $p\in Q$ define
\[h(p)=\min\{\alpha<\theta\mid p\in Q_\alpha\}.\]
The condition (iii) holds because the sets $Q_\alpha$ are upwards closed.
The condition (i)~holds because every set $\{p\in Q\mid h(p)=\alpha\}$ is a~subset
of the leaf-linked set~$Q_\alpha$ and is therefore leaf-linked.
By using~\ref{def:upsoft}~\ref{def:upsoftb} of~\autoref{def:wsolf} for every $\alpha,\beta<\theta$ define
\[g(\alpha,\beta)=\min\set{\gamma<\theta}{\forall p\in Q_\alpha\,\forall q\in Q_\beta\,[p\parallel q\Rightarrow\exists r\in Q_\gamma\,(r\le p,q)]}.\]
Assume $p,q\in Q$ and $p\parallel q$.
Then $p\in Q_{h(p)}$ and $q\in Q_{h(q)}$ and so there is $r\in Q_{g(h(p),h(q))}$
such that $r\le p$ and $r\le q$.
By the definition of~$h$, $h(r)\le g(h(p),h(q))$ and hence (ii)~holds.

\noindent$\ref{up-softc}\Rightarrow\ref{up-softa}$ for regular~$\theta$.
By (i) and~(iii) for every $\alpha<\theta$ the set
$Q_\alpha=\set{p\in Q}{h(p)\le\alpha}$ upwards closed leaf-linked and
$Q=\bigcup_{\alpha<\theta}Q_\alpha$ is a~dense subset
of~$\Pbb$.
We verify~\ref{def:upsoft}~\ref{def:upsoftb} of~\autoref{def:wsolf}.
Fix $\alpha,\alpha'\le\theta$.
Define \[\alpha^*=\sup\{g(\xi,\eta)\mid\xi,\eta\le\max\{\alpha,\alpha'\}\}.\]
Note that $\alpha^*<\theta$ because $\theta$~is regular.
Assume $p\in Q_\alpha$, $q\in Q_{\alpha'}$, and $p\parallel q$.
Then $g(h(p),h(q))\le\alpha^*$ because $h(p)\le\alpha$ and $h(q)\le\alpha'$.
By~(ii) there is $r\le p,q$ such that $h(r)\le g(h(p),h(q))\le\alpha^*$ and
hence $r\in Q_{\alpha^*}$.

\noindent$\ref{up-softb}\Rightarrow\ref{up-softc}$ for $\theta=\omega$.
Note that $\set{p\in Q}{h(p)\le n}$ is a~finite union of leaf-linked sets
$\set{p\in Q}{h(p)=k}$ for $k\le n$ and is therefore leaf-linked.
\end{proof}

The following lemma presents another result similar to the previous one.

\begin{lemma}\label{mup-soft}
Let $\Por$ be forcing notion and let $\theta$ be a~cardinal number. Assume that $\Pbb$ is $\theta$-\mup-soft-linked. Then there is a~dense set $Q\subseteq\Pbb$ and there are 
functions $h\colon Q\to\theta$ and $g\colon\theta\times\theta\to\theta$ with the
following properties:
\begin{enumerate}[label=\rm(\roman*)]
\item
For every $\alpha<\theta$, the set\/ $\{p\in Q\mid h(p)\le\alpha\}$ is
leaf-linked.
\item
$\forall p,q\in Q$
$[p\parallel q\Rightarrow
\exists r\in Q$ $(r\le p,q$ and $h(r)\le g(h(p),h(q)))]$.
\item
$\forall p,q\in Q\,(q\le p\Rightarrow h(q)\ge h(p))$.
\end{enumerate}

If $\theta$ is regular, the converse holds.
\end{lemma}
\begin{proof}
 Let $\langle Q_\alpha\mid\alpha<\theta\rangle$ be an increasing sequence of
upwards closed subsets of~$\Pbb$ witnessing that $\Pbb$~is 
$\theta$-\mup-soft-linked.
The set $Q=\bigcup_{\alpha<\theta}Q_\alpha$ is a~dense subset of~$\Pbb$.
For $p\in Q$ define
\[h(p)=\min\set{\alpha<\theta}{p\in Q_\alpha}.\]
Condition~(iii) holds because the sets $Q_\alpha$ are upwards
closed.
Condition (i)~holds because $\{p\in Q\mid h(p)\le\alpha\}=Q_\alpha$.
By using~\ref{def:rupsoft}~\ref{def:rupsoftb} of~\autoref{def:wsolf} for every $\alpha,\beta<\theta$ define
\[g(\alpha,\beta)=\min\set{\gamma<\theta}{\forall p\in Q_\alpha\,\forall q\in Q_\beta\,[p\parallel q\Rightarrow\exists r\in Q_\gamma\,(r\le p,q)]}.\]
Assume $p,q\in Q$ and $p\parallel q$.
Then $p\in Q_{h(p)}$ and $q\in Q_{h(q)}$ and so there is $r\in Q_{g(h(p),h(q))}$
such that $r\le p$ and $r\le q$.
By definition of~$h$, $h(r)\le g(h(p),h(q))$ and hence (ii)~holds.

For the converse, let $Q_\alpha=\set{p\in Q}{h(p)\le\alpha}$ for $\alpha<\theta$.
By (i) and~(iii), $\langle Q_\alpha\mid\alpha<\theta\rangle$ is an increasing
sequence of upwards closed sets and
$Q=\bigcup_{\alpha<\theta}Q_\alpha$ is a~dense subset of~$\Pbb$.
We now verify~\ref{def:rupsoft}~\ref{def:rupsoftb} of~\autoref{def:wsolf}.
Fix $\alpha,\alpha'\le\theta$.
Define \[\alpha^*=\sup\{g(\xi,\eta)\mid\xi,\eta\le\max\{\alpha,\alpha'\}\}.\]
Notice first that $\alpha^*<\theta$ because $\theta$~is regular.
Assume $p\in Q_\alpha$, $q\in Q_{\alpha'}$, and $p\parallel q$.
Then $g(h(p),h(q))\le\alpha^*$ because $h(p)\le\alpha$ and $h(q)\le\alpha'$.
By~(ii) there is $r\le p,q$ such that $h(r)\le g(h(p),h(q))\le\alpha^*$ and
hence $r\in Q_{\alpha^*}$.
\end{proof}

Combining the above theorem with~\autoref{mup-soft} yields:

\begin{corollary}
 Let $\theta$ be a regular cardinal. Then $\theta$-\mup-soft-linked matches $\theta$-\up-soft-linked.
 \qed
\end{corollary}

\section{Not adding predictors}

The purpose of this section is to deal with the preservation of $\efrak$  along FS iterations of $\theta$-soft-linked posets, which is based on Brendle's ideas to prove the consistency of $\efrak<\eubd$ in~\cite[Sec.4.3]{BrendlevasionI}. We shall prove, in fact, that $\theta$-soft-linked posets do not add predictors, so they do not increase $\efrak$ (see~\autoref{thm:preva}).  We also provide various examples of $\sigma$-soft-linked posets in this section.

\subsection{A brief summary of preservation theory}
\

We begin with some basic definitions but necessary to develop our goodness property:

We say that $\Rbf=\la X, Y, {\sqsubset}\ra$ is a \textit{relational system} if it consists of two non-empty sets $X$ and $Y$ and a relation ${\sqsubset}$.
\begin{enumerate}[label=(\arabic*)]
    \item A set $F\subseteq X$ is \emph{$\Rbf$-bounded} if $\exists\, y\in Y\ \forall\, x\in F\colon x \sqsubset y$. \smallskip
    \item A set $D\subseteq Y$ is \emph{$\Rbf$-dominating} if $\forall\, x\in X\, \exists\, y\in D\colon x \sqsubset y$. \smallskip
    \item Let $M$ be a set. Say that $x$ is \emph{$\Rbf$-unbounded over $M$} if $\forall\, y \in  Y\cap M\colon \neg(x \sqsubset y)$.
\end{enumerate}

We associate two cardinal characteristics with this relational system $\Rbf$: 
\begin{align*}
    \bfrak(\Rbf) & :=\min\set{|F|}{ F\subseteq X  \text{ is }\Rbf\text{-unbounded}} \text{ the \emph{unbounding number of $\Rbf$}, and}\\
    \dfrak(\Rbf) & :=\min\set{|D|}{ D\subseteq Y \text{ is } \Rbf\text{-dominating}} \text{ the \emph{dominating number of $\Rbf$}.}
\end{align*}

We now introduce two important concepts before providing some examples of $\sigma$-soft-linked posets. These concepts were originally introduced by Judah and Shelah~\cite{JS}, they were improved by Brendle~\cite{Br}, and generalized by the author and Mej\'ia~\cite[Sect.~4]{CM}.


\begin{definition}\label{def:Prs}
We say that $\Rbf=\langle X,Y,\sqsubset\rangle$ is a \textit{generalized Polish relational system (gPrs)} if
\begin{enumerate}[label=\rm(\arabic*)]
\item $X$ is a perfect Polish space,\smallskip
\item $Y=\bigcup_{e\in \Omega}Y_e$ where $\Omega$ is a non-empty set and, for some Polish space $Z$, $Y_e$ is non-empty and analytic in $Z$ for all $e\in \Omega$, and\smallskip
\item \label{def:Prsc}${\sqsubset}=\bigcup_{n<\omega}{\sqsubset_{n}}$ where $\seq{{\sqsubset_{n}}}{ n<\omega}$  is an increasing sequence of closed subsets of $X\times Y$ such that, for any $n<\omega$ and for any $y\in Y$,
$(\sqsubset_{n})^{y}=\set{x\in X}{x\sqsubset_{n}y }$ is closed nowhere dense.
\end{enumerate}
If $|\Omega|=1$, we just say that $\Rbf$ is a \emph{Polish relational system (Prs)}.
\end{definition}

\begin{remark}\label{Prsremark}
By~\autoref{def:Prs}~\ref{def:Prsc}, $\Cbf_{\Mwf(X)} \leqT \Rbf$. Therefore, $\bfrak(\Rbf)\leq \non(\Mwf)$ and $\cov(\Mwf)\leq\dfrak(\Rbf)$.
\end{remark}

\begin{definition}\label{def:good}
A poset $\Por$ is \textit{$\theta$-$\Rbf$-good} if, for any $\Por$-name $\dot{h}$ for a member of $Y$, there is a non-empty set $H\subseteq Y$ (in the ground model) of size ${<}\theta$ such that, for any $x\in X$, if $x$~is $\Rbf$-unbounded over  $H$ then $\Vdash x\not\sqsubset \dot{h}$.

We say that $\Por$ is \textit{$\Rbf$-good} if it is $\aleph_1$-$\Rbf$-good. Notice that $\theta<\theta_0$
implies that any $\theta$-$\Rbf$-good poset is $\theta_0$-$\Rbf$-good. Also, if $\Por \lessdot\Qor$ and $\Qor$~is $\theta$-$\Rbf$-good, then $\Por$ is $\theta$-$\Rbf$-good.
\end{definition}

The preceding definition is a standard property associated with preserving $\bfrak(\Rbf)$ small and $\dfrak(\Rbf)$ large after forcing extensions. 

Below, we present some examples of good posets that will be used in the proof of~\autoref{FSitPb}:

\begin{example}\label{exm:good}
We indicate the type of posets that are good for the Prs of Cicho\'n's diagram.
  \begin{enumerate}[label=(\arabic*)]
     \item\label{exm:gooda}  The relational system $\baire:=\la\omega^\omega,\omega^\omega,\leq^*\ra$ is Polish. All of the forcing presented in~\autoref{exmsoft} are examples of $\baire$-good sets. More generally, $\sigma$-$\Fr$-linked posets are $\baire$-good.  (see~\cite{mejvert,BCM}). In particular, $\sigma$-soft-linked poset is $\omega^\omega$-good by~\autoref{SoftFr}.
     
     \item\label{exm:goodb}  Define $\Omega_n:=\set{a\in [2^{<\omega}]^{<\aleph_0}}{\Lb(\bigcup_{s\in a}[s])\leq 2^{-n}}$ (endowed with the discrete topology) and put $\Omega:=\prod_{n<\omega}\Omega_n$ with the product topology, which is a perfect Polish space. For every $x\in \Omega$ denote 
     \[N_{x}:=\bigcap_{n<\omega}\bigcup_{m\geq n}\bigcup_{s\in x(m)}[s],\] which is clearly a Borel null set in $2^{\omega}$.
       
    Define the Prs $\Cn:=\la \Omega, 2^\omega, \sqsubset\ra$ where $x\sqsubset z$ iff $z\notin N_{x}$. Recall that any null set in $2^\omega$ is a subset of $N_{x}$ for some $x\in \Omega$, so $\Cn$ and $\Cbf_\Nwf^\perp$ are Tukey-Galois equivalent. Hence, $\bfrak(\Cn)=\cov(\Nwf)$ and $\dfrak(\Cn)=\non(\Nwf)$.
    
 Any $\mu$-centered poset is $\mu^+$-$\Cn$-good (\cite{Br}). In particular, $\sigma$-centered posets are $\Cn$-good.

 \item Let $\Mbf := \la 2^\omega,\Ior\times 2^\omega, \sqsubm \ra$ where
     \[x \sqsubm (I,y) \text{ iff }\forall^\infty n\colon x\frestr I_n \neq y\frestr I_n.\]
     This is a Polish relational system and $\Mbf\eqT \Cbf_\Mwf$ (see~ \cite{blass}).

     Note that, whenever $M$ is a transitive model of $\thzfc$, $c\in 2^\omega$ is a Cohen real over $M$ iff $c$ is $\Mbf$-unbounded over $M$.

     \item\label{exm:goodc} Kamo and Osuga~{\cite{KO}} define a gPrs with parameters $\varrho,\rho\in\omega^\omega$, which we denote by $\Lc^*(\varrho,\rho)$. We review its properties for the paper's purpose. Assume that $\varrho>0$ and $\rho\geq^* 1$.

\begin{enumerate}[label = \rm (\alph*)]
    \item\label{KOa} $\Lc^*(\varrho,\rho)\leqT \Lc(\varrho,\rho)$~\cite[Lem.~4.21]{CM}.
    \item\label{KOd} Any $\theta$-centered poset is $\theta^+$-$\Lc^*(\varrho,\rho)$-good~\cite[Lem.~4.24]{CM}.
\end{enumerate}
  \end{enumerate}
\end{example}

The Prs defined below is inspired by~\cite[Sec.~4.3]{BrendlevasionI}.

\begin{definition}
A pair $\pi=(\seq{A_k}{k\in\omega}, \seq{\pi_k}{k\in\omega})$ is called \textit{generalized predictor} iff for all $k\in\omega$,
\begin{itemize}
    \item $A_k\subseteq[k,\omega)$ is finite,\smallskip

    \item $\pi_k$ is a function with $\dom\pi_k=\set{\sigma\in\omega^{<\omega}}{\mathrm{lh}(\sigma)\in A_k}$, and \smallskip

    \item $\ran\pi_k\subseteq[\omega]^{<\omega}$.
\end{itemize}

Let $\Pi$ be the collection of generalized predictors. For $\pi\in\Pi$ and $f\in\omega^\omega$, we say   \[\pi\textit{\ predicts\ }f~ (\textrm{denoted by\ } f\in^{\textbf{pr}}\pi)\textrm{\ iff\ } \forall^\infty k\in\omega\,\exists\ell\in A_k(f(\ell)\in \pi_k(f{\upharpoonright}\ell));\]
otherwise, $f$~\emph{evades}~$\pi$. Define the relational system $\mathbf{E}=\la\omega^\omega,\Pi,\in^{\textbf{pr}}\ra$. The original definition of predicting is a~particular instance of this notion (in the case $A_k=\{\ell_k\}$, $\ell_k<\ell_{k+1}$ and $|\pi_k(\sigma)|=1$ for $\sigma\in\omega^{\ell_k}$).
\end{definition}

\subsection{Some examples of \texorpdfstring{$\sigma$}{}-soft-linked posets}\label{exmsoft}
\ 

Now we look at instances of $\sigma$-soft-linked posets. To illustrate the next example, it is necessary to present the definition of ~\emph{$(b,h)$-ED (eventually different real) forcing}, denoted by $\Eor^h_b$, which was introduced by Kamo and Osuga in~\cite{KO} and which has been extensively investigated in \cite{CM, BCM, BCM2}.  This forcing is used to increase $\balc_{b,h}$ (the \emph{anti-localization cardinal}, see e.g.~\cite{CMlocalc} for the definition and more).

\begin{definition}\label{EDforcing}
Fix $b\colon\omega\to(\omega+1)\menos\{0\}$ and $h\in\omega^\omega$ such that $\lim_{i\to+\infty}\frac{h(i)}{b(i)}=0$ (when $b(i)=\omega$, interpret $\frac{h(i)}{b(i)}$ as $0$).
      Define the \emph{$(b,h)$-ED forcing} $\Eor^h_b$ as the poset whose conditions are of the form $p=(s,\varphi)$ such that, for some $m:=m_{p}<\omega$, 
         \begin{enumerate}
            \item $s\in\Seq_{<\omega}(b)=\bigcup_{n<\omega}\prod_{k<n}b(k)$, $\varphi\in\Swf(b,m\cdot h)$, and\smallskip
            \item $m\cdot h(i)<b(i)$ for every $i\geq|s|$.
         \end{enumerate}
      Ordered by $(t,\psi)\leq(s,\varphi)$ iff $s\subseteq t$, $\forall i<\omega(\varphi(i)\subseteq\psi(i))$ and $t(i)\notin\varphi(i)$ for all $i\in|t|\menos|s|$. For $s\in\Seq_{<\omega}(b)$ and $m<\omega$ put    
      \[E^h_b(s,m):=\set{(t,\varphi)\in\Eor^h_b}{t=s\text{\ and }m_{(t,\varphi)}\leq m}.\]
      Denote $\Eor_b:=\Eor^1_b$, $\Eor:=\Eor_\omega$, $E_b(s,m):=E^1_b(s,m)$, and $E(s,m):=E_\omega(s,m)$.

     It is not hard to see that $\Eor_b^h$ is $\sigma$-linked. Even more, whenever $b\geq^*\omega$, $\Eor_b^h$ is $\sigma$-centered.

     When $h\geq^*1$, $\Eor_b^h$ adds an eventually different real\footnote{More generally, $\Eor_b^h$ adds a real $e\in\prod b$ such that $\forall^\infty i<\omega\,(e(i)\notin\varphi(i))$ for any $\varphi\in\Scal(b,h)$ in the ground model.} in $\prod b$.  
\end{definition}

\begin{example}
 Let $b,h$ be as in~\autoref{EDforcing}. Then $\Eor^h_b$ is $\sigma$-soft-linked. In particular, $\Eor$ is $\sigma$-soft-linked.
\end{example}
In point of fact, it is clear that $\bigcup_{(s,m)}E^h_b(s,m)$ is dense in $\Eor^h_b$ and by \cite[Lem.~3.8]{BCM}, it is known that $E^h_b(s,m)$ is Fr-linked, so by~\autoref{thm:a1} it is leaf-linked. Therefore,~\ref{maindefba} and \ref{maindefbb} of \autoref{maindef}~\ref{maindefb} are met. 

Let $m, m'\in\omega$, and $s, s'\in\Seq_{<\omega}(b)$. Wlog assume that $s'$ is longer than $s$. Next, let $s^*:=s'$ and $m^*:=m+m'$. It can be proved that for any $(t,\psi)\in E^h_b(s,m)$ and $(t',\psi')\in E^h_b(s',m')$ compatible then there is $(t^*,\psi^*)\in E^h_b(s^*,m^*)$ extending $(t,\psi)$ and $(t',\psi')$. Hence, ~\ref{maindefbc} of~\autoref{maindef}~\ref{maindefb} holds.

We now introduce a forcing to increase $\blc_{b,h}$, which typically adds \emph{a slalom}. This forcing was introduced by Brendle and Mej\'ia~\cite{BrM}.

\begin{definition}
For $b,h\in\omega^\omega$ such that
$\forall i<\omega\,(b(i)>0$ and $h(i)\le b(i))$,
$\lim_{i\to\infty}h(i)=\infty$, and
$\lim_{i\to\infty}h(i)/b(i)=0$  define the localization forcing $\LOCor_{b,h}$ by \[\LOCor_{b,h}:=\set{(p,n)}{p\in\Swf(b,h),\, n<\omega\textrm{\ and\ }\exists m<\omega\,\forall i<\omega\,(|p(i)|\leq m)},\]
ordered by $(p',n')\leq(p,n)$ iff $n\leq n'$, $p'{\upharpoonright}n=p$, and $\forall i<\omega(p(i)\subseteq p'(i))$. 
\end{definition}

\begin{example}
 $\LOCor_{b,h}$ is $\sigma$-soft-linked.
\end{example}
Indeed, for $s\in\Scal_{<\omega}(b,h)=\bigcup_{k\in\omega}\prod_{n<k}[b(n)]^{\le h(n)}$, and $m<\omega$, define
\[
 L_{b,h}(s,m):=\set{(p,n)\in\LOCor_{b, h}}{s\subseteq p,\, n=|s|, \textrm{\ and\ }\forall i<\omega\,(|p(i)|\leq m)}.  
\] 
Notice that  $\bigcup_{(s,m)\in\Scal_{<\omega}(b,h)\times\omega}\LOCor_{b,h}(s,m)$ is dense in $\LOCor_{b,h}$ and  also note that $L_{b,h}(s,m)$ is Fr-linked by~\cite[Rem.~3.31]{mejvert} (see also~\cite[Lem.~4.10]{Car4E}). Since $L_{b,h}(s,m)$ is Fr-linked, by~\autoref{thm:a1}, it is leaf-linked. Hence, ~\ref{maindefba} and \ref{maindefbb} of \autoref{maindef}~\ref{maindefb} hold.

Let $m, m'\in\omega$, and $s, s'\in\Scal_{<\omega}(b,h)$. Wlog assume that $s'$ is longer than $s$. Put $s^*:=s'$ and $m^*:=m+m'$. If $(p,n)\in L_{b,h}(s,m)$ and $(p',n')\in L_{b,h}(s',m')$ with $(p,\varphi)\,\|\,(p',\varphi')$, then there is a $(p^*,n^*)\in L_{b,h}(s^*,m^*)$ such that $(p^*,n^*)\leq (p,n), (p',n')$. This shows ~\ref{maindefbc} of~\autoref{maindef}~\ref{maindefb}.

We now proceed to present the meager forcing (denoted by $\Mor$) proposed by Judah and Shelah~\cite{JS}. They demonstrated by them that $\Mor$ does not add dominating reals and increases $\non(\Mwf)$.
  
\begin{definition}
For $\eta\subseteq2^{<\omega}$ denote $n_\eta=\sup\set{|s|}{s\in \eta}$. Let $T$ be the set of the all $\eta\subseteq2^{<\omega}$ such that 
\begin{enumerate}[label=\rm(\alph*)]
    \item $n_\eta\in\omega$,\smallskip

    \item if $v\in\eta$ and $k<|v|$, then $v{\upharpoonright}k\in\eta$; and\smallskip

    \item if $v\in\eta$ and $|v|<n_\eta$, then $v{}^{\smallfrown}0\in\eta$ or $v{}^{\smallfrown}1\in\eta$.
\end{enumerate}

For $\eta\in T$ let $[\eta]=\set{x\in2^\omega}{x{\upharpoonright}n_\eta\in\eta}\subseteq2^\omega$ and $\tilde \eta=\set{u\in\eta}{|u|=n_\eta}$. Define the \emph{meager forcing $\Mor$} as the set \[\Mor=\set{(\eta, H)}{\eta\in T\text{\ and\ }H\subseteq[\eta] \text{\ is  finite} }\] ordered by $(\nu, K)\leq(\eta, H)$ iff $\tilde \eta=\nu\cap2^{n_\nu}$ and $H\subseteq K$.  It is not hard to see that $\Mor$ is $\sigma$-centered. Moreover, $\Mor$ adds a Cohen real.

Let 
\[\Mor^*=\set{(\eta, H)\in\Mor}{\forall v\in\tilde\eta\,(|H\cap[v]|=1)}.\]
Notice that $\Mor^*$ is dense in $\Mor$. 
\end{definition}
We prove that $\Mor^*$ is $\sigma$-soft-linked, so $\Mor$ is $\sigma$-soft-linked as well. First note that~\ref{maindefba} of \autoref{maindef} holds and \ref{maindefbb} of \autoref{maindef} holds by~\autoref{uf:M}.

Our focus in what follows is on proving~\autoref{uf:M}. The following linkedness is stronger than Fr-linkedness. 

\begin{definition}[{\cite{GMS}}]
Let $\Por$ be a poset, and let $D\subseteq \pts(\omega)$ be a non-principal ultrafilter. Say that $Q\subseteq\Por$ is \emph{$D$-$\lim$-linked} if there is a function $\lim^{D}\colon Q^\omega\to \Por$ and a $\Por$-name $\dot D'$ of an ultrafilter extending $D$ such that, for any $\bar q = \la q_i:\, i<\omega\ra \in Q^\omega$,
\[{\lim}^{D}\, \bar q \Vdash \dot{W}(\bar{q})\in \dot D'.\]     
A set $Q\subseteq \Por$ has \emph{uf-$\lim$-linked} if it is $D$-\textrm{lim}-linked for any ultrafilter $D$. 

For an infinite cardinal $\theta$, the poset $\Por$ is~\emph{uniformly $\mu$-$D$-$\lim$-linked} if $\Por = \bigcup_{\alpha<\theta}Q_\alpha$ where each $Q_\alpha$ is $D$-$\lim$-linked and the $\Por$-name $\dot D'$ above mentioned only depends on $D$ and not on $Q_\alpha$, although we have different limits for each $Q_\alpha$. When these $Q_\alpha$ are \emph{uf-$\lim$-linked}, we say that $\Por$ is \emph{uniformly $\mu$-uf-$\lim$-linked}.
\end{definition}

\begin{remark}\label{rem:uf}
For instance, random forcing is $\sigma$-uf-linked (see~\cite[Lem.~3.29]{mejvert}), but it may not be $\sigma$-uf-$\lim$-linked (cf.~\cite[Rem.~3.10]{BCM}). It is clear that any uf-$\lim$-linked set $Q\subseteq \Por$ is uf-linked, which implies $\Fr$-linked.  
\end{remark}

Our objective is to demonstrate that $\Mor^*$ is uniformly $\sigma$-uf-$\lim$-linked, witnessed by
\[M_\eta:=M^*_\eta=\set{(\eta',H)\in\Mor^*}{\eta'=\eta}\]
for $\eta\in T$.  Let $D$ be an ultrafilter on $\omega$, and $\bar{p}=\seq{p_n}{n\in\omega}$ be a sequence in $M^*_\eta$ with $p_n = (\eta,H_n)$. Let $m_n:=|\eta\cap2^{n_\eta}|$. Considering the lexicographic order $\lhd$ of $\cantor$, let $\set{x_{n,k}}{k<m_n}$ be a $\lhd$-increasing enumeration of $H_n$. Next, find $a_0\in D$ and $m_0<\omega$ such that $H_{n}=\set{x_{n,k}}{k\in m_0}$ for all $n\in a_0$. For each $k<m_0$, define $x_k=\lim_n^D x_{n,k}$ in $\cantor$ by
\[x_k(i)=s\textrm{\ iff\ }\set{n\in a_0}{x_{n,k}(i)=s}\in D,\]
which matches the topological $D$-limit. 
Then, the $D$-limit of $H_n$ can be defined as $H:=\set{x_k}{k<m_0}$ and $\lim^D \bar p := (\eta,H)$. It is clear that this limit is in $M^*_\eta$.

\begin{theorem}\label{uf:M}
The poset $\Mor^*$ is uniformly $\sigma$-uf-$\lim$-linked:
For any ultrafilter $D$ on $\omega$, there is a $\Mor^*$-name of an ultrafilter $\dot D'$ on $\omega$ extending $D$ such that,
for any $\eta\in T$, and $\bar p\in M_{\eta}^{\omega}$, $\lim^D \bar p \Vdash \dot W(\bar p)\in \dot D'$. 
\end{theorem}

It suffices to show the following to establish the previous theorem:

\begin{clm}
      Assume $M<\omega$, $\set{\eta_k}{k<M}\subseteq T$, $\set{\bar{p}^k}{k<M}$ is such that each $\bar{p}^k=\seq{ p_{k,n}}{n<\omega}$ is a sequence in $M_{\eta_k}$, $q_k$ is the $D$-limit of $\bar{p}^k$ for each $k<M$, and $q\in\Mor^*$ is stronger than every $q_k$. If $a\in D$ then $q$ forces that $a\cap\bigcap_{k<M}\dot{W}(\bar{p}^k)\neq\emptyset$.
\end{clm}
\begin{proof}
 Write $p_{k,n}=(s_k,t_k,H_{k,n})$, $q_k=(s_k,t_k,H_{k})$ where each $H_k = \set{x^k_j}{j<m_{0}^k}$ is the $D$-limit of $H_{k,n} = \set{x^{k,n}_j}{j<m_{0}^k}$ (increasing $\lhd$-enumeration) with $m_{0}^k\leq m_k$, for all $n\in a$ (wlog). Assume that $(\eta, H)\leq q\leq q_k$ in $\Mor$ for all $k<M$. Then $\tilde \eta_k=\eta\cap2^{n_\eta}$ and $H_k\subseteq H$, therefore \[a_k=\set{n\in\omega}{\forall j<m_{0}^{k}\,( x^{k,n}_j{\restriction}n_{\eta}=x_j^k{\restriction}n_{\eta})}\]
is in~$D$. Hence, $a\cap\bigcap_{k<M}b_k\neq\emptyset$, so choose $n\in a\cap\bigcap_{k<M}b_k$, and put $r:=(\eta,H')$ where $H':=H\cup \bigcup_{k<M}H_{k,n}$.  Then $H_{k,n}\subseteq[\eta]$ because for every $j<m_0^k$,
$x^{k,n}_j{\restriction}n_{\eta}=x_j^k{\restriction}n_{\eta}\in\tilde\eta$, $H'\subseteq[\eta]$, so $r$ is a condition in $\Mor$ and $r$ is stronger than $q$ and $p_{n,k}$ for any $k<M$. Since $\Mor^*$ is dense in $\Mor$, we can find $r'\leq r$, which implies that $r'$ forces $n\in a\cap\bigcap_{k<M}\dot{W}(\bar{p}^k)$.     
\end{proof}

We have seen that $M_\eta$ are uf-lim-linked, so by~\autoref{rem:uf} $M_\eta$ are Fr-linked. Hence, by~\autoref{thm:a1}, $M_\eta$ are leaf-linked, so the condition ~\ref{maindefba} of~\autoref{maindef}~\ref{maindefb} holds. It remains to prove~\ref{maindefbc} of~\autoref{maindef}~\ref{maindefb}. To see this,
let $(\eta,H)\in M_{\eta}$ and $(\eta',H')\in M_{\eta'}$. Wlog assumes that $\eta'$ is longer than $\eta$. So put $\eta^*=\eta'$. It is not difficult to see that, if $(\eta,H)\in M_{\eta}$ and $(\eta',H')\in M_{\eta'}$ with $(\eta,H)\,\|\,(\eta',H')$, then there is a $(\eta^*,H^*)\in M_{\eta^*}$ such that $(\eta^*,H^*)\leq (\eta,H), (\eta',H')$. 

To state the next example, we need the definition of the Dirichlet–Minkowski forcing, due to 
Bukovsk\'y, Rec\l aw, and Repick\'y~\cite[Sec.~7]{BRR}. They used this forcing notion to increase the uniformity of the collection of the Dirichlet sets ($D$-sets) of $[0,1]$. The forcing Dirichlet–Minkowski forcing is defined by \[\DMor:=\set{(s,F)}{s\in\omega^{<\omega}\textrm{\ and\ }F\in[[0,1]]^{<\omega}}.\] 
We order $\DMor$ by $(s,F)\leq(t,H)$ iff $t\subseteq t$, $H\subseteq F$, and $\forall i\in\dom(s\smallsetminus t)\,\forall x\in H\,(\parallel s(i)\cdot x\parallel<\frac{1}{i+1})$. 
The poset $\DMor$ is $\sigma$-centered, since for $s\in\omega^{<\omega}$ the set 
\[M_{s}=\set{(t,A)\in\DMor}{t=s}\]
is centered and $\bigcup_{s\in\omega^{<\omega}}M_{s}=\DMor$.

We claim that $\DMor$ is $\sigma$-soft-linked. In order to see this, we shall verify that 
the sets \[M_{s,m}=\set{(t,F)\in\DMor}{t=s\textrm{\ and\ }|F|\leq m},\]
$s\in\omega^{<\omega}$ and $m\in\omega$, witness that $\DMor$ is $\sigma$-soft-linked. We leave to the reader to check the~\ref{maindefbb} and~\ref{maindefbc} of~\autoref{maindef}~\ref{maindefb}. We show that~\ref{maindefba} of~\autoref{maindef}~\ref{maindefb}. Because of~\autoref{thm:a1}, it suffices to prove that the sets $M_{s,m}$ are Fr-linked.

As in the previous example, we proceed to show a version of Fr-linked.  We shall demonstrate that the sets $M_{s,m}$ (as defined above) witness that $\DMor$ is uniformly $\sigma$-uf-$\lim$-linked. For an ultrafilter $D$ on $\omega$,  and $\bar{p}=\seq{p_n}{n\in\omega}\in M_{s,m}$, we show how to define $\lim^D\bar p$. Let $p_n=(t,F_n) \in M_{s,m}$. Let $\set{x_{n,k}}{k<m_n}$ be a $\lhd$-increasing enumeration of $F_n$ where $m_n \leq m$. Next find an unique  $m_*\leq m$ such that $A:=\set{n\in\omega}{m_n=m_*}\in D$. For each $k<m_*$, define $x_k:=\lim_n^D x_{n,k}$ in $[0,1]$ where $x_k$ is the unique member of $[0,1]$ such that $\set{n\in A}{x_{n,k}= x_k} \in D$ (this coincides with the topological $D$-limit). Therefore, we can think of $F:=\set{x_k}{k<m_*}$ as the $D$-limit of $\seq{F_n}{n<\omega}$, so we define $\lim^D \bar p:=(t,F)$. Note that $\lim^D \bar p \in M_{s,m}$. 

\begin{theorem}\label{ufQf}
The poset $\DMor$ is uniformly $\sigma$-uf-$\lim$-linked: 
If $D$ is an ultrafilter on $\omega$, 
then there is a  $\DMor$-name of an ultrafilter $\dot D'$ on $\omega$ extending $D$ such that, for any $s\in\omega^{<\omega}$, $m\in\omega$ and $\bar p\in M_{s,m}^\omega$, $\lim^D \bar p \Vdash W(\bar p)\in \dot D'$.
\end{theorem}

Proof of the following is sufficient to prove the above theorem.

\begin{clm}
      Assume $M<\omega$, $\set{(s_k, m_k)}{k<M}\subseteq\omega^{<\omega}\times \omega$, $\set{\bar{p}^k}{k<M}$ is such that each $\bar{p}^k=\seq{ p_{k,n}}{n<\omega}$ is a sequence in $M_{s_k,m_k}$, $q_k$ is the $D$-limit of $\bar{p}^k$ for each $k<M$, and $q\in\DMor$ is stronger than every $q_k$. Then, for any $a\in D$, there are $n<\omega$ and $q'\leq q$ stronger than $p_{k,n}$ for all $k<M$ (i.e.\ $q'$ forces $a\cap\bigcap_{k<M}\dot{W}(\bar{p}^k)\neq\emptyset$).
\end{clm}
\begin{proof}
Write $p_{k,n}=(s_k,F_{k,n})$, $q_k=(s_k,F_{k})$ where each $F_k = \set{x^k_j}{j<m_{*,k}}$ is the $D$-limit of $F_{k,n} = \set{x^{k,n}_j}{j<m_{*,k}}$ (increasing $\lhd$-enumeration) with $m_{*,k}\leq m_k$. Assume that $q = (s,F)\leq q_k$ in $\DMor$ for all $k<M$. Let 
\[U_k:=\largeset{\seq{x_j}{j<m_{*,k}}}{
 \forall j<m_{*,k}\,\forall i\in\dom(s\smallsetminus s_k)\,\bigg( \parallel s(i)\cdot x_j\parallel<\frac{1}{i+1}\bigg)}\]
which is an open neighborhood of $\seq{x^k_j}{j<m_{*,k}}$ in $[0,1]^{m_{*,k}}$. 
Then 
\[b_k := \largeset{n<\omega }{  \forall j<m_{*,k}\,\forall i\in\dom(s\smallsetminus s_k)\colon  \parallel s(i)\cdot x_{j}^{k,n}\parallel<\frac{1}{i+1}}\in D.\]
Hence, $a\cap\bigcap_{k<M}b_k\neq\emptyset$, so choose $n\in a\cap\bigcap_{k<M}b_k$ and put $q'=(s,F')$ where $F':=F\cup \bigcup_{k<M}F_{k,n}$. This is a condition in $\DMor$ because $|F'|\leq |F|+\sum_{k<M}m_{*,k}$. Furthermore, $q'$ is stronger than $q$ and $p_{n,k}$ for any $k<M$.
\end{proof}

In order to establish the consistency of $\efrak<\eubd$, Brendle~\cite{BrendlevasionI} introduced the following forcing notion: 
\begin{definition}\label{Brpr}
 Given a $b\in\omega^\omega$ define the forcing $\Pror_b$ as the poset whose conditions are of the form $p=(A,\seq{ \pi_n}{n\in A}, F)$ such that 
\begin{enumerate}[label=(\roman*)]
    \item $A\in[\omega]^{<\aleph_0}$,\smallskip
    \item $\forall n\in A$, $\pi_n\colon\prod_{m<n}b(m)\to b(n)$, and \smallskip
    \item $F$ finite set of functions and $\forall f\in F\,\exists k\leq\omega\,(f\in\prod_{n<k}b(n))$, i.e., $F$ is a finite subset of $\Seq_{<\omega}(b)\cup\Seq(b)$.
\end{enumerate}
The order is defined by 
$(A',\pi',F')\le(A,\pi,F)$ iff
\begin{enumerate}[label=(\alph*)]

\item\label{Brpr:a}
$A\subseteq A'$, $\pi\subseteq\pi'$,\smallskip
\item\label{Brpr:b}
if $A'\ne A$ and $A\ne\emptyset$, then $\max A<\min(A'\smallsetminus A)$,\smallskip
\item\label{Brpr:c}
$\forall f\in F\,\exists g\in F'\,(f\subseteq g)$, and\smallskip
\item\label{Brpr:d}
$\forall f\in F\,\forall n\in (A'\smallsetminus A)\cap\dom(f)\,(\pi'_n(f{\restriction}n)=f(n))$.
\end{enumerate}

    
    
Furthermore, $\Pror_b$ is $\sigma$-centered (and thus in particular ccc). In particular, if $(A',\pi',F')\le(A,\pi,F)$,
$f_0,f_1\in F$,
$n\in\dom(f_0)\cap\dom(f_1)\smallsetminus A$,
$f_0{\restriction}n=f_1{\restriction}n$, and
$f_0(n)\ne f_1(n)$, then $n\notin A'$.
For any two conditions $p_i=(A_{p_i},\pi_{p_i},F_{p_i})$, $i\in2$,
in~$\Pror_b$,
\begin{align*}
p_0\parallel p_1\Leftrightarrow{}
&(A_{p_0}\cup A_{p_1},\pi_{p_0}\cup\pi_{p_1},F_{p_0}\cup F_{p_1})
\text{ is a~condition}\\*
&\text{and this condition  is a lower bound of $p_0$ and $p_1$ in $\Pror_b$}
\tag{$\boxtimes$}\label{Brbox}\\
\Leftrightarrow{}
&A_{p_0}=A_{p_1}\text{ and }\pi_{p_0}=\pi_{p_1}\text{ or }\\
&\exists i\in2\
A_{p_i}\subsetneq A_{p_{1-i}}\text{ and }\pi_{p_i}\subsetneq\pi_{p_{1-i}}\text{ and }
\max A_{p_i}<\min(A_{p_{1-i}}\smallsetminus A_{p_i})\\
&\phantom{\exists i\in2}\text{ and }
\forall f\in F_{p_i}\ \forall n\in(A_{p_{1-i}}\smallsetminus A_{p_i})\cap\dom(f)\
(\pi_{p_{1-i}})_n(f{\restriction}n)=f(n).
\end{align*}
\end{definition}

Let $G$ be a $\Pror_b$-generic filter over $V$. In $V[G]$, define 
\[\pi_{\gen}:=\bigcup\largeset{\seq{ \pi_n}{n\in A}}{\exists A,F( (A,\seq{ \pi_n}{n\in A},F) \in G)}.\]

Then $\pi_{\gen}\in\Pi_b$ and for every $f\in \prod b\cap V\,(f\sqsubset^{\textbf{pr}}\pi_{\gen})$.

Our aim now is to show that any poset of the form $\Pror_b$ is $\sigma$-soft-linked.

\begin{lemma}\label{lem:Prb}
For any $b\in\omega^\omega$, $\Pror_b$ is $\sigma$-soft-linked.  
\end{lemma}
\begin{proof}
Let $h\colon\Pror_b\to\omega$ be defined by $h(p)=\max(A_p\cup\{|F_p|\})$. One can see easily that is a~height function where $p=(A_p,\pi_p,F_p)$. We will verify that the pair $\la\Pbb,h\ra$ satisfies the conditions of the club property~(\ref{Brtrevol}):

(I): Assume that $P=\{p_n\mid n<\omega\}$ is a~decreasing sequence of conditions
in~$\Pror_b$ such that $\{h(p_n)\mid n<\omega\}$ is eventually constant.
Then the sequence $\{(A_{p_n},\pi_{p_n},|F_{p_n}|)\mid n<\omega\}$ is eventually
constant and eventually equal to some $(A,\pi,m)$.
Therefore, one can easily observe that there is $F\in[\omega^\omega]^m$ such that
$(A,\pi,F)$ is a~lower bound of~$P$.
This proves~(I).

(III): If $p,q\in\Pror_b$ are compatible, then
$r=(A_p\cup A_q,\pi_p\cup\pi_q,F_p\cup F_q)$ is a~lower bound of $p$ and~$q$
and $h(r)\le h(p)+h(q)$. Therefore (III)~holds.

It remains to prove~(II). To this end, for $m,k\in\omega$ denote $Q_m=\set{p\in\Pror_b}{h(p)\le m}$ and
$R_{m,k}=\set{p{\restriction}k}{p\in Q_m}$ where
$p{\restriction}k=
(A_p\cap k,\pi_p{\restriction}k,\set{f{\restriction}k}{f\in F_p})$
for $p\in\Pror_b$.
Every set $R_{m,k}$ is finite and
for every $p\in Q_m$, $p\le p{\restriction}k$ and
$h(p{\restriction}k)\le h(p)\le m$.
Thus, it is easy to see that (II) of~(\ref{Brtrevol}) implies (II') where 

\begin{enumerate}
\item[(II')]
$\forall m\in\omega\,\forall P\in[\Pror_b]^{<\omega}\smallsetminus\{\emptyset\}\,\exists k\ge m\,\forall p\in Q_m\,(p\perp P\Rightarrow(p{\restriction}k)\perp P)$.
\end{enumerate}
Therefore, to conclude the proof it is enough to check~(II'). To see this, let  $P\in[\Pror_b]^{<\omega}\smallsetminus\{\emptyset\}$. Find
$k\ge m$ such that $\bigcup_{q\in P}A_q\subseteq k$.
We show that $k$~is as required.
Let $p\in Q_m$ be such that $p\perp P$ and let $q\in P$ be arbitrary.
\begin{enumerate}[label=\rm(\roman*)]
    \item If $\pi_p\cup\pi_q$ is not a~function, then $(p{\restriction}k)\perp q$
because $\dom(\pi_q)=A_q\subseteq k$.
\end{enumerate}
Let us assume that $\pi_p\cup\pi_q$ is a~function.
Denote
$p'=(A_p\cup A_q,\pi_p\cup\pi_q,F_p\cup F_q)$ and
$p''=(A_p\cup A_q,\pi_p\cup\pi_q,(F_p{\restriction}k)\cup F_q)$.
Then $p',p''\in\Pror_b$, $p'\le p''$, and by~using~(\ref{Brbox}) of~\autoref{Brpr},
\begin{align*}
&p\perp q\Leftrightarrow\text{$p'$ is not a~lower bound of $p$ and~$q$;}\\*
&(p{\restriction}k)\perp q\Leftrightarrow
\text{$p''$ is not a~lower bound of $p{\restriction}k$ and~$q$.}
\end{align*}
We assume that $p\perp q$ and hence,
\begin{equation*}
\text{$p'\nleq p$ or $p'\nleq q$.}
\end{equation*}
Note that the relations $(p'\nleq p$ or $p'\nleq q)$ and
$(p''\nleq p{\restriction}k$ or $p''\nleq q)$ cannot be justified by
the violation of~\ref{Brpr:c} of~\autoref{Brpr}.

\begin{enumerate}[label=\rm(\roman*)]
\setcounter{enumi}{1}
\item Assume that $(p'\nleq p$ or $p'\nleq q)$ is justified by the violation
of~$(\beta)$,
i.e.,
\begin{align*}
&\text{$A_p\cup A_q\ne A_p$ and $A_p\ne\emptyset$ and
$\max A_p\ge\min((A_p\cup A_q)\smallsetminus A_p)$ or}\\*
&\text{$A_p\cup A_q\ne A_q$ and $A_q\ne\emptyset$ and
$\max A_q\ge\min((A_p\cup A_q)\smallsetminus A_q)$.}
\end{align*}
Since all mentioned sets are subsets of~$k$, this property gives
$(p''\nleq p{\restriction}k$ or $p''\nleq q)$ by violation of~\ref{Brpr:c} of~\autoref{Brpr}.
Therefore $(p{\restriction}k)\perp q$.
\end{enumerate}

\begin{enumerate}[label=\rm(\roman*)]
\setcounter{enumi}{2}
\item Assume that $(p'\nleq p$ or $p'\nleq q)$ is justified by a~violation
of~\ref{Brpr:b} of~\autoref{Brpr}, i.e., one of the following conditions holds:
\begin{align*}
&\exists f\in F_p\ \exists n\in((A_p\cup A_q)\smallsetminus A_p)\cap\dom(f)\
(\pi_q)_n(f{\restriction}n)\ne f(n)\text{ or}\\*
&\exists f\in F_q\ \exists n\in((A_p\cup A_q)\smallsetminus A_q)\cap\dom(f)\
(\pi_p)_n(f{\restriction}n)\ne f(n).
\end{align*}
\end{enumerate}
Since all mentioned sets are subsets of~$k$ by replacing $F_p$ with
$F_p{\restriction}k$ we get $(p''\nleq p{\restriction}k$ or $p''\nleq q)$
by violation of~\ref{Brpr:b} of~\autoref{Brpr}.
Therefore $(p{\restriction}k)\perp q$ also in this last case.
\end{proof}

\subsection{Forcing the evasion number to be small}
\ 

We now prove the main theorem of this section. This theorem generalizes all of the examples shown in~\autoref{exmsoft}.

\begin{theorem}\label{thm:preva}
  Let $\Por$ be a ccc forcing notion. Assume that $\Por$~is $\theta$-$\mathrm{soft}$-linked. Then $\Por$ is $\theta^+$-$\mathbf{E}$-good.
\end{theorem}

\begin{proof}
Since $\Por$ is  $\theta$-$\mathrm{soft}$-linked, by employing~\autoref{solfch}, there is a dense subset
$Q=\bigcup_{\alpha<\theta}Q_\alpha\subseteq\Pbb$ with all sets $Q_\alpha$ being
soft-linked and there is $g\colon\theta\times Q\to\theta$ fulfills~(\ref{solfch:Beq}).

We prove that for every $\Pbb$-name $\dot\pi$ such that
$\Vdash_\Pbb\dot\pi\in\Pi$ there is a non-empty set $H\subseteq \Pi$ (in the ground model) of size ${<}\theta$ such that, for any $f\in\omega^\omega$, if $f$ is $\mathbf{E}$-unbounded over  $H$ then $\Vdash_\Por f\not\in^{\mathbf{pr}}\dot\pi$. Without loss of generality, we can assume that
$\Pbb=Q=\bigcup_{\alpha<\theta}Q_\alpha$.

Suppose that $\Vdash_\Pbb\dot\pi\in\Pi$ and that $\dot\pi$ is of the form
$(\langle\dot A_k\mid k\in\omega\rangle,\langle\dot\pi_k\mid k\in\omega\rangle)$.
For every $k\in\omega$ fix a~maximal antichain
$\{p_k^j\mid j\in\omega\}$ deciding $\dot A_k$.
For every $k,j\in\omega$ let
$A_k^j\in[\omega\smallsetminus k]^{<\omega}$
be such that $p_k^j\Vdash\dot A_k=A_k^j$ and let
$\{\ell_{k,j,i}\mid i<a_{k,j}\}$ be the increasing enumeration of~$A_k^j$.
Since $Q_\alpha$ is soft-linked,
for every $k\in\omega$ we can find $n_k^\alpha\in\omega$ such that
\begin{equation*}
\forall p\in Q_\alpha\
\exists j<n_k^\alpha\
(p\parallel p_k^j).
\end{equation*}
Denote $\eta_k^{j,\alpha}=g(\alpha,p_k^j)$;
then by applying~(\ref{solfch:Beq}),
\begin{equation*}
\forall p\in Q_\alpha\
\exists j<n_k^\alpha\
\exists q\in Q_{\eta_k^{j,\alpha}}\
(q\le p,p_k^j).
\tag{$\varotimes_1$}
\label{boxone}
\end{equation*}

For every $k,j\in\omega$ by induction on $i<a_{k,j}$ we define
\begin{align*}
&p_{k,\sigma}^{j,\tau}\in\Pbb
\text{ and }
\eta_{k,\sigma}^{j,\tau,\alpha}<\theta
&&\text{for $\sigma\in\omega^{\ell_{k,j,i}}$ and $\tau\in\omega^{i+1}$,}\\
&n_{k,\sigma}^{j,\tau,\alpha}\in\omega 
&&\text{for $\sigma\in\omega^{\ell_{k,j,i}}$ and $\tau\in\omega^i$}
\end{align*}
(in fact, we will need $\eta_{k,\sigma}^{j,\tau,\alpha}$ and
$n_{k,\sigma}^{j,\tau,\alpha}$ only for $j<n_k^\alpha$ and only for finitely
many bounded sequences~$\tau$; ``how many~$\tau$'' depends on~$\alpha$).

\paragraph*{Case $i=0$}
For $\sigma\in\omega^{\ell_{k,j,0}}$ let
$\{p_{k,\sigma}^{j,\langle n\rangle}\mid n\in\omega\}$ be a~maximal antichain
of conditions below $p_k^j$ deciding the value of $\dot\pi_k(\sigma)$.
Since all $Q_\beta$'s are soft-linked we can find
$n_{k,\sigma}^{j,\emptyset,\alpha}\in\omega$ such that
\begin{equation*}
\forall p\in Q_{\eta_k^{j,\alpha}}\
p\le p_k^j\Rightarrow
\exists n<n_{k,\sigma}^{j,\emptyset,\alpha}\
(p\parallel p_{k,\sigma}^{j,\langle n\rangle}).
\end{equation*}
Denote $\eta_{k,\sigma}^{j,\langle n\rangle,\alpha}=
g(\eta_k^{j,\alpha},p_{k,\sigma}^{j,\langle n\rangle})$;
then by utilizing~\eqref{solfch:Beq},
\begin{equation*}
\forall p\in Q_{\eta_k^{j,\alpha}}\
p\le p_k^j\Rightarrow
\exists n<n_{k,\sigma}^{j,\emptyset,\alpha}\
\exists q\in Q_{\eta_{k,\sigma}^{j,\langle n\rangle,\alpha}}\
(q\le p,p_{k,\sigma}^{j,\langle n\rangle}).
\tag{$\varotimes_2$}
\label{boxtwo}
\end{equation*}

\paragraph*{Case $i+1$}
For every $\sigma\in\omega^{\ell_{k,j,i+1}}$ and $\tau\in\omega^{i+1}$ let
$\{p_{k,\sigma}^{j,\tau^\frown\langle n\rangle}\mid n\in\omega\}$ be
a~maximal antichain of conditions below
$p_{k,\sigma{\restriction}\ell_{k,j,i}}^{j,\tau}$ deciding the value
of~$\dot\pi_k(\sigma)$.
Since all $Q_\beta$'s are soft-linked we can find
$n_{k,\sigma}^{j,\tau,\alpha}\in\omega$ such that
\begin{equation*}
\forall p\in Q_{\eta_{k,\sigma{\restriction}\ell_{k,j,i}}^{j,\tau,\alpha}}\
p\le p_{k,\sigma{\restriction}\ell_{k,j,i}}^{j,\tau}\Rightarrow
\exists n<n_{k,\sigma}^{j,\tau,\alpha}\
(p\parallel p_{k,\sigma}^{j,\tau^\frown\langle n\rangle}).
\end{equation*}
Denote
$\eta_{k,\sigma}^{j,\tau^\frown\langle n\rangle,\alpha}=
g(\eta_{k,\sigma{\restriction}\ell_{k,j,i}}^{j,\tau,\alpha},
p_{k,\sigma}^{j,\tau^\frown\langle n\rangle})$;
then by using~\eqref{solfch:Beq},
\begin{equation*}
\forall p\in Q_{\eta_{k,\sigma{\restriction}\ell_{k,j,i}}^{j,\tau,\alpha}}\
p\le p_{k,\sigma{\restriction}\ell_{k,j,i}}^{j,\tau}\Rightarrow
\exists n<n_{k,\sigma}^{j,\tau,\alpha}\
\exists q\in Q_{\eta_{k,\sigma}^{j,\tau^\frown\langle n\rangle,\alpha}}\
(q\le p,p_{k,\sigma}^{j,\tau^\frown\langle n\rangle}).
\tag{$\varotimes_3$}
\label{boxthird}
\end{equation*}
This finishes the definitions of all
$p_{k,\sigma}^{j,\tau}$,
$\eta_{k,\sigma}^{j,\tau,\alpha}$, and 
$n_{k,\sigma}^{j,\tau,\alpha}$.

\medskip

Hence for $k,j\in\omega$, $i<a_{k,j}$,
$\sigma\in\omega^{\ell_{k,j,i}}$, and $\tau\in\omega^{i+1}$ the
condition~$p_{k,\sigma}^{j,\tau}$ decides~$\dot\pi_k(\sigma)$;
let $A_{k,\sigma}^{j,\tau}\in[\omega]^{<\omega}$ be such that
$p_{k,\sigma}^{j,\tau}\Vdash \dot\pi_k(\sigma)=A_{k,\sigma}^{j,\tau}$.

Now, for every $\alpha<\theta$ we define
$\pi^\alpha=(\langle A^\alpha_k\mid k\in\omega\rangle,
\langle\pi_k^\alpha\mid k\in\omega\rangle)\in\Pi$ as follows:
Let $A_k^\alpha=\bigcup_{j<n_k^\alpha}A_k^j$ and let
$\pi_k^\alpha\colon\bigcup_{n\in A_k^\alpha}\omega^n\to[\omega]^{<\omega}$
be defined by
\begin{align*}
\pi_k^\alpha(\sigma)=\bigcup\big\{A_{k,\sigma}^{j,\tau}\mid{}
&\exists j<n_k^\alpha\
\exists i<a_{k,j}\\
&\sigma\in\omega^{\ell_{k,j,i}}
\text{ and }
\tau\in\omega^{i+1}
\text{ and }
\forall m\le i\
(\tau(m)<
n_{k,\sigma{\restriction}\ell_{k,j,m}}^{j,\tau{\restriction}m,\alpha})\big\}.
\end{align*}
Let $f\in\omega^\omega$ and assume that $f\not\in^{\mathbf{pr}}\pi^\alpha$ for all $\alpha<\theta$. So it suffices to prove that:
\begin{clm}
$\Vdash_{\Por}f\not\in^{\mathbf{pr}}\dot\pi$.
\end{clm}

Towards a~contradiction assume that there are $p\in\Pbb$ and $k_0\in\omega$
such that
$p\Vdash``\forall k\ge k_0\,\exists\ell\in\dot A_k\,(f(\ell)\in\dot\pi_k(f{\restriction}\ell))$''.
Take $\alpha<\theta$ such that $p\in Q_\alpha$.
By the choice of~$f$ there is $k\ge k_0$ such that
$\forall\ell\in A_k^\alpha\,(f(\ell)\notin\pi_k^\alpha(f{\restriction}\ell))$.
By~using~\eqref{boxone}, there is $j<n_k^\alpha$ and there is a~condition 
$q_k^{j,\alpha}\in Q_{\eta_k^{j,\alpha}}$ such that $q_k^{j,\alpha}\le p,p_k^j$.
Then by~using~\eqref{boxtwo} there is
$n_0<n_{k,f{\restriction}\ell_{k,j,0}}^{j,\emptyset,\alpha}$ and
$q_{k,f{\restriction}\ell_{k,j,0}}^{j,\langle n_0\rangle,\alpha}\in
Q_{\eta_{k,f{\restriction}\ell_{k,j,0}}^{j,\langle n_0\rangle,\alpha}}$
such that
$q_{k,f{\restriction}\ell_{k,j,0}}^{j,\langle n_0\rangle,\alpha}\le
q_k^{j,\alpha},p_{k,f{\restriction}\ell_{k,j,0}}^{j,\langle n_0\rangle}$.

Using~\eqref{boxthird} by induction on $i<a_{k,j}$ define
$\tau\in\omega^{a_{k,j}}$ and a~decreasing sequence of conditions
$\{q_{k,f{\restriction}{\ell_{k,j,i}}}^{j,\tau{\restriction}(i+1),\alpha}\mid
i<a_{k,j}\}$
below $q_{k,f{\restriction}\ell_{k,j,0}}^{j,\langle n_0\rangle,\alpha}$
such that
$\tau(0)=n_0$,
$q_{k,f{\restriction}{\ell_{k,j,i}}}^{j,\tau{\restriction}(i+1),\alpha}\in
Q_{\eta_{k,f{\restriction}{\ell_{k,j,i}}}^{j,\tau{\restriction}(i+1),\alpha}}$,
and
$q_{k,f{\restriction}{\ell_{k,j,i}}}^{j,\tau{\restriction}(i+1),\alpha}\le
p_{k,f{\restriction}{\ell_{k,j,i}}}^{j,\tau{\restriction}(i+1)}$.
(Assume that $\tau{\restriction}(i+1)$ is defined.
By~(\ref{boxone}) we find
$\tau(i+1)\le n_{k,f{\restriction}\ell_{k,j,i+1}}^{j,\tau{\restriction}(i+1),\alpha}$
and
$q_{k,f{\restriction}\ell_{k,j,i+1}}^{j,\tau{\restriction}(i+2),\alpha}\in
Q_{\eta_{k,f{\restriction}\ell_{k,j,i+1}}^{j,\tau{\restriction}(i+2),\alpha}}$
such that
$q_{k,f{\restriction}\ell_{k,j,i+1}}^{j,\tau{\restriction}(i+2),\alpha}\le
q_{k,f{\restriction}\ell_{k,j,i}}^{j,\tau{\restriction}(i+1),\alpha},
p_{k,f{\restriction}\ell_{k,j,i+1}}^{j,\tau{\restriction}(i+2)}$.)

Let $m=a_{k,j}-1$ and $q=q_{k,f{\restriction}\ell_{k,j,m}}^{j,\tau,\alpha}$.
It follows that $\forall i<a_{k,j}\,(q\le p_{k,f{\restriction}\ell_{k,j,i}}^{j,\tau{\restriction}(i+1)}\le p_k^j)$
where
$p_{k,f{\restriction}\ell_{k,j,i}}^{j,\tau{\restriction}(i+1)}\Vdash
\dot\pi_k(f{\restriction}\ell_{k,j,i})=
A_{k,f{\restriction}\ell_{k,j,i}}^{j,\tau{\restriction}(i+1)}$ and
$p_k^j\Vdash\dot A_k=A_k^j$.
Clearly $A_k^j\subseteq A_k^\alpha$
and for every $i<a_{k,j}$,
$A_{k,f{\restriction}\ell_{k,j,i}}^{j,\tau{\restriction}(i+1)}\subseteq
\pi^\alpha(f{\restriction}\ell_{k,j,i})$.
Therefore we get
$q\Vdash\forall\ell\in\dot A_k\,(f(l)\notin\pi_k^\alpha(f{\restriction}\ell)\supseteq\dot\pi_k(f{\restriction}\ell))$
which is a contradiction.
\end{proof}

By a similar argument as in the proof of~\autoref{thm:preva} and by~applying~\autoref{mup-soft}, we have:

\begin{theorem}
Every ccc\ $\theta$-\mup-soft-linked forcing\/ $\Por$~is $\theta^+$-$\mathbf{E}$-good.    
\end{theorem}

To conclude this section, we will demonstrate Brendle's consistency of ~$\efrak<\eubd$ by utilizing the results we have obtained so far.

In order to prove it, we begin with some notation:

We begin with some notation, for two posets $\Por$ and $\Qor$, we write $\Por\subsetdot\Qor$ when $\Por$ is a complete suborder of $\Qor$, i.e.,\ the inclusion map from $\Por$ into $\Qor$ is a complete embedding.

\begin{definition}[Direct limit]\label{def:limdir}
We say that $\la\Por_i:\, i\in S\ra$ is a \emph{directed system of posets} if $S$ is a directed preorder and, for any $j\in S$, $\Por_j$ is a poset and $\Por_i\subsetdot\Por_j$ for all $i\leq_S j$.

For such a system, we define its \emph{direct limit} $\limdir_{i\in S}\Por_i:=\bigcup_{i\in S}\Por_i$ ordered by
\[q\leq p \sii \exists\, i\in S\,( p,q\in\Por_i\text{ and }q\leq_{\Por_i} p).\]
\end{definition}

FS iterations add Cohen reals, which is sometimes considered as a limitation of the method. 

\begin{corollary}[{\cite[Cor.~2.7]{CM22}}]\label{Cohenlimit}
Let $\pi$ be a limit ordinal of uncountable cofinality and let $\Por_\pi=\seq{\Por_\alpha,\Qnm_\alpha}{  \alpha<\pi}$ be a FS iteration of non-trivial posets. If $\Por_\pi$ is $\cf(\pi)$-cc then it forces $\pi\leqT\Mbf$. In particular, $\Por_\pi$ forces $\non(\Mwf)\leq\cf(\pi)\leq\cov(\Mwf)$. 
\end{corollary}

Good posets are preserved along FS iterations as follows.

\begin{theorem}[{\cite[Sec.~4]{BCM2}}]\label{Comgood}
Let $\seq{\Por_\xi,\Qnm_\xi}{\xi<\pi}$ be a FS iteration such that, for $\xi<\pi$, $\Por_\xi$ forces that $\Qnm_\xi$ is a non-trivial $\theta$-cc $\theta$-$\Rbf$-good poset. 
Let $\set{\gamma_\alpha}{\alpha<\delta}$ be an increasing enumeration of $0$ and all limit ordinals smaller than $\pi$ (note that $\gamma_\alpha=\omega\alpha$), and for $\alpha<\delta$ let $\dot c_\alpha$ be a $\Por_{\gamma_{\alpha+1}}$-name of a Cohen real in $X$ over $V_{\gamma_\alpha}$. 

Then $\Por_\pi$ is $\theta$-$\Rbf$-good. Moreover,  if $\pi\geq\theta$ then $\Cbf_{[\pi]^{<\theta}}\leqT\Rbf$, $\bfrak(\Rbf)\leq\theta$ and $|\pi|\leq\dfrak(\Rbf)$.
\end{theorem}

The following theorem illustrates the effect of iterating forcing $\Por_b$ on Cicho\'n's diagram.

\begin{theorem}\label{FSitPb}
Let $\pi$ be an ordinal of uncountable cofinality such that $|\pi|^{\aleph_0}=|\pi|$. 
The FS iteration of $\Por_b$ of length $\pi$ (i.e. the FS iteration $\seq{\Por_\alpha,\Qnm_\alpha}{ \alpha<\pi}$, where each $\Qnm_\alpha$ is a $\Por_\alpha$-name of $\Por_b$) forces 
\begin{enumerate}[label=\rm(\arabic*)]
    \item $\cfrak=|\pi|$;
    \item $\Cbf_{\NAwf}\eqT\Cbf_\Nwf^\perp\eqT\omega^\omega\eqT\mathbf{E}\eqT\Cbf_{[\R]^{<\aleph_1}}$; and
    \item $\Cbf_\Mwf\eqT\mathbf{E}_b\eqT\cf(\pi)$ for any $b\in((\omega+1)\smallsetminus2)^\omega$.
\end{enumerate}
In particular, it is forced~\autoref{FigthmB}.
\begin{figure}[ht!]
\centering
\begin{tikzpicture}[scale=1.2]
\small{
\node (aleph1) at (-1,3) {$\aleph_1$};
\node (addn) at (0.5,3){$\add(\Nwf)$};
\node (covn) at (0.5,7){$\cov(\Nwf)$};
\node (nonn) at (9.5,3) {$\non(\Nwf)$} ;
\node (cfn) at (9.5,7) {$\cof(\Nwf)$} ;
\node (addm) at (4,3) {$\add(\Mwf)$} ;
\node (covm) at (6.9,3) {$\cov(\Mwf)$} ;
\node (nonm) at (4,7) {$\non(\Mwf)$} ;
\node (cfm) at (6.9,7) {$\cof(\Mwf)$} ;
\node (b) at (4,5) {$\bfrak$};
\node (d) at (6.9,5) {$\dfrak$};
\node (c) at (11,7) {$\cfrak$};
\node (nonna) at (
2.5,4) {$\non(\NAwf)$};
\node (cove) at (1.8,5) {$\efrak$};
\node (eu) at (2.5,6.2) {$\eubd$};
\draw (aleph1) edge[->] (addn)
      (addn) edge[->] (covn)
      (covn) edge [->] (nonm)
      (nonm)edge [->] (cfm)
      (cfm)edge [->] (cfn)
      (cfn) edge[->] (c);

\draw
   (addn) edge [->]  (addm)
   (addm) edge [->]  (covm)
   (covm) edge [->]  (nonn)
   (cove) edge [->]  (eu)
   (eu) edge [->]  (nonm)
   (nonn) edge [->]  (cfn);
\draw (addm) edge [->] (b)
      (b)  edge [->] (nonm);
\draw (covm) edge [->] (d)
      (d)  edge[->] (cfm);
\draw (b) edge [->] (d);

\draw  (addn)edge [->] (nonna);     
\draw  (addn) edge [->] (cove);
          
\draw (cove) edge [line width=.15cm,white,-] (covm)
      (cove) edge [->] (covm);
\draw (nonna) edge [line width=.15cm,white,-] (nonn)
      (nonna) edge [->] (nonn);

\draw (nonna) edge [line width=.15cm,white,-] (eu)
      (nonna) edge [->] (eu);

\draw[color=sug,line width=.05cm] (1.8,7.5)--(1.8,5.6); 
\draw[color=sug,line width=.05cm] (1.8,5.6)--(4.6,5.6);
\draw[color=sug,line width=.05cm] (4.6,5.6)--(4.6,2.5);
\draw[color=sug,line width=.05cm] (6,7.5)--(6,4.4); 
\draw[color=sug,line width=.05cm] (6,4.4)--(7.5,4.4); 
\draw[color=sug,line width=.05cm] (7.5,4.4)--(7.5,2.5);

\draw[circle, fill=yellow,color=yellow] (1.1,5) circle (0.4);
\draw[circle, fill=yellow,color=yellow] (5.25,6) circle (0.4);
\draw[circle, fill=yellow,color=yellow] (8.25,6) circle (0.4);
\node at (1.1,5) {$\aleph_1$};
\node at (5.25,6) {$\cf(\pi)$};
\node at (8.25,6) {$|\pi|$};
}
\end{tikzpicture}
\caption{Cichon's diagram after adding $\pi$-many predictors reals with $\Pror_b$, where $\pi$ has uncountable cofinality and $|\pi|^{\aleph_0}=|\pi|$.}
\label{FigthmB}
\end{figure}
\end{theorem}
\begin{proof}
\begin{enumerate}[label=\rm($\boxplus$)]
    \item\label{prof:cl} Let $\set{b_\xi}{\xi<\pi}$ be an enumeration of all the nice $\Por_\alpha$-name for all members of $((\omega+1)\smallsetminus2)^\omega$.
\end{enumerate}
We define the iteration at each $\alpha$ as follows:
\[\Qnm_\alpha:=\begin{array}{ll}
         \Pror_{b_\alpha}  & \text{if $\dot b_{\alpha}$ is a $\Pbb_{\alpha+1}$-name for a function in $((\omega+1)\smallsetminus2)^\omega$,}\\
    \end{array}\]
$\Por_{\alpha+1}=\Por_\alpha\ast\Qnm_\alpha$, and $\Por_\alpha=\limdir_{\xi<\alpha} \Por_\xi$ if $\alpha$ limit. 
First, note that $\Por_\pi$ satisfies ccc and forces $\cfrak=|\pi|$. Next, in view of \autoref{Cohenlimit}, $\Por_\pi$ forces $\pi\leqT\Cbf_\Mwf$. Since $\Cbf_\Mwf\leqT\mathbf{E}_b$ (by~\autoref{Rem:Rb}), it suffices to prove that $\mathbf{E}_b\leqT\pi$, that is, in $V_\pi$, there are maps $\Psi_-\colon\prod b\to\pi$ and $\Psi_+\colon\pi\to\Pi_b$ such that, for any $x\in\baire$, and for any $\alpha<\pi$, if $\Psi_-(x)\leq \alpha$, then $x=^{\textbf{pr}}\Psi_+(\alpha)$. To this end, for each $\alpha<\pi$, denote by $\pi^{\gen}_\alpha$ the predictor generic real over $V_\alpha$ added by $\Qnm_\alpha$.

By ccc, there is $\alpha_b<\pi$ such that $b\in V_{\alpha_b}$. Next, for $x\in\prod b\cap V_\pi$, we can find $\alpha_b\leq\alpha_x<\pi$ such that $x\in V_{\alpha_x}$, so put $\Psi_-(x)=\alpha_x$. On the other hand, for $\alpha<\pi$, when $\alpha\geq\alpha_b$, by~\ref{prof:cl} there is an $\xi_\alpha$ such that $b=b_{\xi_\alpha}$, so define $\Psi_+(\alpha)=\pi^{\gen}_{\xi_\alpha}$; otherwise, $\Psi_-(x)$ can be anything. It is clear that $(\Psi_-,\Psi_+)$ is the required Tukey connection.
 
Using the fact that $\Por_\pi$ is obtained by FS iteration $\seq{\Por_\alpha,\Qnm_\alpha}{ \alpha<\pi}$ and that all its iterands are $\baire$-good, $\Cn$-good and $\mathbf{E}$-good (see~\ref{exm:gooda}~and~\ref{exm:goodb} of \autoref{exm:good} and~\autoref{thm:preva}, respectively), so by employing~\autoref{Comgood}, $\Por_\pi$ forces $\Cbf_{[\baire]^{<\aleph_1}}\leqT\baire$, $\Cbf_{[\baire]^{<\aleph_1}}\leqT\Cn$ and $\Cbf_{[\baire]^{<\aleph_1}}\leqT\mathbf{E}$, respectively. Lastly, the opposite Tukey connections can be proved easily.

It remains to see that $\Por$ forces $\Cbf_{\NAwf}\eqT\Cbf_{[\R]^{<\aleph_1}}$. We just prove the nontrivial Tukey connection, i.e., $\Cbf_{[\R]^{<\aleph_1}}\leqT\Cbf_{\NAwf}$. Since all of the iterands of the iteration are $\sigma$-centered, by using~\autoref{exm:good}~\ref{exm:goodc}, they are $\Lc^*(\varrho,\rho)$-good, so by applying~\autoref{Comgood}, $\Por_\pi$ forces $\Cbf_{[\R]^{<\aleph_1}}\leqT\Lc^*(\varrho,\rho)$. Moreover, by using~(a) of~\autoref{exm:good}~\ref{exm:goodc}, $\Por_\pi$ forces $\Cbf_{[\R]^{<\aleph_1}}\leqT\Lc(\varrho,\rho)$. Finally, by the fact that $\Lc(\varrho,\id_\omega)\leqT\Cbf_{\NAwf}$ (see~\cite[Thm.~8.3]{CMlocalc}) holds in ZFC, we obtain that $\Por_\pi$ forces $\Cbf_{[\R]^{<\aleph_1}}\leqT\Cbf_{\NAwf}$. This finishes the proof of the theorem.
\end{proof}

\section{Open problems}

One of our primary motivations for introducing the concept of soft-linkedness was the aim to force the constellation of~\autoref{cichon+e}. The point is that we do not know if restricted versions of the forcing discussed in~\autoref{exmsoft} behave well with our preservation theory for $\efrak$.  Taking into account the latter, the following is of interest:

\begin{problem}\label{restsolft}
Could we modify our soft-linkedness property so that restrictions of the forcing dealt with in~\autoref{exmsoft} do not add predictors?   
\end{problem}

Concerning~previous question, recently, Yamazoe~\cite{Ye} introduced a new linkedness property (see~\autoref{b7}), which keeps the evasion number small and also can be considered on restrictions of forcings. He used this notion to prove the constellation of~\autoref{cichon+e}. 

\begin{definition}[{\cite{Ye}}]\label{b7}
Let $\Por$ be a poset, and $D\subseteq \pts(\omega)$ be a non-principal ultrafilter. A set $Q\subseteq\Por$ is \emph{sufficiently $D$-$\lim$-linked} if there is a function $\lim^{D}\colon Q^\omega\to \Por$ satisfying $(\bigstar_n)$ for all $n\in\omega$ where $(\bigstar_n)\colon$ 
\begin{equation}
 \parbox{0.8\textwidth}{
 Given $\bar p^j=\seq{p_m^j}{m\in\omega}\in Q^\omega$  for $j<n$ and $r\leq\lim^D\bar p^j$ for all $j<n$,\ $\set{m\in\omega}{r\parallel p_m^j\ \textrm{for all}\ j<n}\in D$.
}
\tag{$\bigstar_n$}
\label{eq:b7}
\end{equation}
Additionally, if $\ran(\lim^D)\subseteq Q$, then $Q$ is \emph{sufficiently closed $D$-$\lim$-linked}. We say that $Q$ is \emph{sufficiently uf-$\lim$-linked} if it is sufficiently $D$-$\lim$-linked for any $D$. 

For an infinite cardinal $\mu$, $\Por$ is \emph{sufficiently $\mu$-$D$-$\lim$-linked}
 if $\Por = \bigcup_{\alpha<\mu}Q_\alpha$ for some sufficiently $D$-$\lim$-linked $Q_\alpha$ ($\alpha<\mu$). When these $Q_\alpha$ are sufficiently uf-$\lim$-linked, we say that $\Por$ is \emph{sufficiently $\mu$-uf-$\lim$-linked}.
\end{definition}

He also proved his notion is related to uf-$\lim$-linked (see~\cite[Lem.~3.28]{Ye}). So we ask:

\begin{problem}
Is there any relationship between the notion of closed-ultrafilter-limit and our notion of soft-linkedness?    
\end{problem}

In case~\autoref{restsolft} has a positive answer, we may force the constellation of~\autoref{cichon+e}.

\begin{figure}[H]
\centering
\begin{tikzpicture}[scale=1.2]
\small{
\node (aleph1) at (-1,3) {$\aleph_1$};
\node (addn) at (0.5,3){$\add(\Nwf)$};
\node (covn) at (0.5,7){$\cov(\Nwf)$};
\node (nonn) at (9.5,3) {$\non(\Nwf)$} ;
\node (cfn) at (9.5,7) {$\cof(\Nwf)$} ;
\node (addm) at (3,3) {$\add(\Mwf)$} ;
\node (covm) at (6.9,3) {$\cov(\Mwf)$} ;
\node (nonm) at (3,7) {$\non(\Mwf)$} ;
\node (cfm) at (6.9,7) {$\cof(\Mwf)$} ;
\node (b) at (3,5) {$\bfrak$};
\node (d) at (6.9,5) {$\dfrak$};
\node (c) at (11,7) {$\cfrak$};
\node (cove) at (1.6,4.2) {$\efrak$};
\node (eu) at (1.6,6) {$\eubd$};
\draw (aleph1) edge[->] (addn)
      (addn) edge[->] (covn)
      (covn) edge [->] (nonm)
      (nonm)edge [->] (cfm)
      (cfm)edge [->] (cfn)
      (cfn) edge[->] (c);

\draw
   (addn) edge [->]  (addm)
   (addm) edge [->]  (covm)
   (covm) edge [->]  (nonn)
   (cove) edge [->]  (eu)
   (eu) edge [->]  (nonm)
   (nonn) edge [->]  (cfn);
\draw (addm) edge [->] (b)
      (b)  edge [->] (nonm);
\draw (covm) edge [->] (d)
      (d)  edge[->] (cfm);
\draw (b) edge [->] (d);

\draw   
        (addn) edge [->] (cove);
      
      
\draw (cove) edge [line width=.15cm,white,-] (nonn)
      (cove) edge [->] (nonn);


\draw[color=sug,line width=.05cm] (-0.5,5.4)--(4.4,5.4);
\draw[color=sug,line width=.05cm] (-0.5,2.5)--(-0.5,7.5);

\draw[color=sug,line width=.05cm] (1.2,2.5)--(1.2,5.4);
\draw[color=sug,line width=.05cm] (2.7,3.5)--(2.7,5.4);

\draw[color=sug,line width=.05cm] (1.2,3.5)--(2.7,3.5);

\draw[color=sug,line width=.05cm] (4.4,2.5)--(4.4,7.5);
\draw[color=sug,line width=.05cm] (1.2,5.4)--(1.2,7.5);
 

\draw[circle, fill=yellow,color=yellow] (0,4.7) circle (0.4);
\draw[circle, fill=yellow,color=yellow] (0,6.2) circle (0.4);
\draw[circle, fill=yellow,color=yellow] (3.8,4.4) circle (0.4);
\draw[circle, fill=yellow,color=yellow] (2.15,4.7) circle (0.4);
\draw[circle, fill=yellow,color=yellow] (3.8,6) circle (0.4);
\draw[circle, fill=yellow,color=yellow] (5.5,6) circle (0.4);
\node at (0,4.7) {$\lambda_1$};
\node at (0,6.2) {$\lambda_2$};
\node at (3.8,4.4) {$\lambda_3$};
\node at (2.15,4.7) {$\lambda_4$};
\node at (3.8,6) {$\lambda_5$};
\node at (5.5,6) {$\lambda_6$};
}
\end{tikzpicture}
\caption{Separation of the left side of Cicho\'n's diagram with $\efrak$.}
\label{cichon+e}
\end{figure}

\begin{remark}
The constellation of~\autoref{cichon+e} was first forced by Goldstern, Kellner, Mej\'ia, and Shelah (unpublished). They used finitely additive measures (FAMs) along FS iterations to construct a poset to force that 
\begin{equation*}\label{cicmaxe}
 \aleph_1<\add(\Nwf)<\cov(\Nwf)<\bfrak<\efrak<\non(\Mwf)<\cov(\Mwf)<\dfrak<\non(\Nwf)<\cof(\Nwf).
\end{equation*}
Though the previous work is unpublished, this was announced in~\cite{GKMSe}.   
\end{remark}

On the other hand, consider random forcing $\Bor$ as the poset whose conditions are trees $T\subseteq2^{<\omega}$ such that $\Lb([T])>0$ where $\Lb$ denotes the Lebesgue measure on $2^\omega$. The order is $\subseteq$.

For $(s,m)\in2^{<\omega}\times\omega$ set
       \[B(s,m):=\{T\in\Bor: [T]\subseteq[s]\text{\ and }2^{|s|}\cdot\Lb([T])\geq 1-2^{-(1+m)}\}.\]
First of all, notice that $\bigcup_{(s,m)\in 2^{{<}\omega}\times\omega}B(s,m)$ is dense in $\Bor$, and by \cite[Lem.~3.29]{mejvert} $B(s,m)$ is Fr-linked  for any $m\in\omega$ and $s\in2^{<\omega}$, so~\ref{maindefba} and \ref{maindefbb} of \autoref{maindef}~\ref{maindefb} hold.

Note that if $B(s,m)$'s work then also $B(k,m)$'s must work
because $B(k,m)$ is the union of finitely many ($2^k$ many) of $B(s,m)$'s.

To see the condition~\ref{maindefbc} of~\autoref{maindef}~\ref{maindefb}, we want to show that given any $(k,m)$ and any $(k',m')$ we find $(k^*,m^*)$ in some way.
Assume that $T\in B(k,m)$ and $T'\in B(k',m')$ are compatible and let
$s\in2^k$, $s'\in2^{k'}$ be such that $[T]\subseteq[s]$ and $[T']\subseteq[s']$.
Without loss of generality assume that $k\le k'$ and hence $s\subseteq s'$.
Then
\[\Lb([T'])\ge2^{-k'}\cdot(1-2^{-(m'+1)})=2^{-k'}-2^{-(m'+k'+1)}\]
and
\begin{align*}
    \Lb([T]\cap[s'])&\ge\Lb([T])-2^{-k'}\cdot(2^{k'-k}-1)\\
    &\ge
2^{-k}\cdot(1-2^{-(m+1)})-2^{-k'}\cdot(2^{k'-k}-1)\\
&=2^{-k'}-2^{-(m+k+1)}.
\end{align*}
It follows that
\[\Lb([T\cap T'])\ge
2^{-k'}-2^{-(m'+k'+1)}-2^{-(m+k+1)}.\]
We do not know whether there is a better estimation but this estimation does not
suffice to conclude that there is a~suitable $(k^*,m^*)$: It may happen that $m+k+1\le k'$ and we get a~negative value. Then the condition~\ref{maindefbc} of~\autoref{maindef}~\ref{maindefb} is unclear. So we ask:

\begin{problem}
Is it possible to obtain another representation for random forcing that allows $\Bor$ to be $\sigma$-soft-linked?
\end{problem}

\subsection*{Acknowledgments}  I wish to express my gratitude to Miroslav Repický for his invaluable constructive input on the core concept of this paper. His feedback significantly enhanced the quality of the initial version, and his assistance in revising the content was greatly appreciated. I am also thankful to Diego Mejía for engaging in numerous insightful discussions related to the subject matter and to Adam Marton for his thorough review and helpful suggestions that contributed to the refinement of our final submission.

{\small
\bibliography{name}
\bibliographystyle{alpha}
}


\end{document}